\documentclass[11pt]{amsart}

\usepackage[T1]{fontenc}
\usepackage{lmodern}
\usepackage{microtype}
\usepackage[a4paper,margin=28mm]{geometry}
\usepackage{amsmath,amssymb,mathtools}
\usepackage{xcolor}
\usepackage{float}
\usepackage{tikz}
\usetikzlibrary{arrows.meta,calc}
\usepackage[colorlinks=true,linkcolor=blue!55!black,citecolor=blue!55!black,urlcolor=blue!55!black]{hyperref}
\hypersetup{
  pdftitle={Asymptotically optimal bracketing covers for anchored boxes with applications to star discrepancy},
  pdfauthor={Kosuke Suzuki},
  pdfkeywords={bracketing number, delta-cover, star discrepancy, quasi-Monte Carlo method, metric entropy}
}

\newtheorem{theorem}{Theorem}[section]
\newtheorem{proposition}[theorem]{Proposition}
\newtheorem{lemma}[theorem]{Lemma}
\newtheorem{corollary}[theorem]{Corollary}
\theoremstyle{remark}
\newtheorem{remark}[theorem]{Remark}

\title[Asymptotically optimal bracketing covers]{Asymptotically optimal bracketing covers for anchored boxes with applications to star discrepancy}
\author[K.~Suzuki]{Kosuke Suzuki}
\address[K.~Suzuki]{Faculty of Science, Yamagata University, 1-4-12 Kojirakawa-machi, Yamagata, 990\mbox{-}8560, Japan}
\email[]{kosuke-suzuki@sci.kj.yamagata-u.ac.jp}
\thanks{The work of K.S.\ is supported by JSPS KAKENHI Grant Numbers 24K06857 and 26K00620.}
\date{\today}

\begin{document}

\begin{abstract}
Bracketing covers and $\delta$-covers provide finite discretizations of the
anchored boxes that define the star discrepancy. Let $N_{[]}(d,\delta)$ and
$N(d,\delta)$ denote the corresponding bracketing and covering numbers. We
prove the lower bounds
\[
  N_{[]}(d,\delta)\ge \lceil \delta^{-d}\rceil,
  \qquad
  N(d,\delta)\ge
  \left\lceil \frac{d!}{d^d}\,\delta^{-d}\right\rceil.
\]
We give two explicit constructions of bracketing covers. For every fixed
$d$, together with the lower bound they imply
$N_{[]}(d,\delta)=(1+o_d(1))\delta^{-d}$ as $\delta\downarrow0$. A first
construction uses box-dependent anisotropic local grids and gives simple
explicit bounds. A second, homothetic logarithmic-shell construction again
attains this coefficient and gives
$\limsup_{d\to\infty}N_{[]}(d,\delta)^{1/d}\le\delta^{-1}+e+O(\delta)$
as $\delta\downarrow0$. Combining these finite estimates with
Gnewuch's general bracketing bound and a Hoeffding--Bernstein chaining
argument shows that, for every
$d,n\in\mathbb N$, there exists an $n$-point set with star discrepancy at
most $2.3463\sqrt{d/n}$. Consequently,
$\lceil5.5052d\varepsilon^{-2}\rceil$ points suffice for star discrepancy at
most $\varepsilon$.
\end{abstract}

\maketitle

\section{Introduction}
\label{sec:introduction}

\subsection{Star discrepancy, the inverse problem, and finite covers}

For $\boldsymbol{x}=(x_1,\ldots,x_d)\in[0,1]^d$, let
\[
  B_{\boldsymbol{x}}=[\boldsymbol{0},\boldsymbol{x})
  =\prod_{j=1}^d[0,x_j),
  \qquad
  V(\boldsymbol{x})=\prod_{j=1}^d x_j.
\]
For an $n$-point set $P\subset[0,1)^d$, define
\[
  \Delta_P(\boldsymbol{x})
  =\frac1n\sum_{\boldsymbol p\in P}
    \mathbf 1_{B_{\boldsymbol{x}}}(\boldsymbol p)
   -V(\boldsymbol{x}),
  \qquad
  D_n^*(P)=\sup_{\boldsymbol{x}\in[0,1]^d}
  |\Delta_P(\boldsymbol{x})|.
\]
The quantity $D_n^*(P)$ is the star discrepancy of $P$. The
Koksma--Hlawka inequality makes it a central quality criterion for
deterministic quasi-Monte Carlo integration
\cite{DickPillichshammer2010,Niederreiter1992}. For every function $f$ of
bounded Hardy--Krause variation,
\[
  \left|
  \int_{[0,1]^d}f(\boldsymbol{x})\,\mathrm d\boldsymbol{x}
  -\frac1n\sum_{\boldsymbol p\in P}f(\boldsymbol p)
  \right|
  \le V_{\mathrm{HK}}(f)D_n^*(P).
\]
Thus a certified bound for $D_n^*(P)$ gives a deterministic integration-error
bound for this class of integrands.

A complementary question is how small the star discrepancy can be for a
given number of points and dimension. Write
\[
  D^*(n,d):=\inf_{|P|=n}D_n^*(P),
\]
and, for $0<\varepsilon<1$, define
\begin{equation}
  n^*(d,\varepsilon)
  :=
  \min\left\{
    n\in\mathbb N:
    D^*(n,d)\le\varepsilon
  \right\}.
  \label{eq:inverse-star-discrepancy}
\end{equation}
The inverse star-discrepancy problem asks for the joint dependence of
$n^*(d,\varepsilon)$ on the dimension $d$ and the prescribed accuracy $\varepsilon$.
Heinrich, Novak, Wasilkowski, and Wo\'zniakowski proved that
$n^*(d,\varepsilon)\le C d\varepsilon^{-2}$ for a universal constant $C$
\cite{HNWW2001}. Subsequent work made the constant explicit and progressively
smaller through finite covers and chaining
\cite{Aistleitner2011,AistleitnerHofer2014,GnewuchHebbinghaus2021,
GPW2021,Weiss2026}. Dick recently proved that the accuracy exponent $2$
is optimal \cite{DickStar2026}. Together with Hinrichs's lower bound
\cite{Hinrichs2004}, which implies the optimality of the dimension exponent $1$,
this establishes the individual optimality of both exponents.
A bound $D^*(n,d)\le C\sqrt{d/n}$ yields
$n^*(d,\varepsilon)\le\lceil C^2d\varepsilon^{-2}\rceil$,
so improving $C$ gives a quantitative improvement in the inverse problem.

A standard route to such existence bounds is to replace the continuum of
anchored boxes in the definition of $D_n^*(P)$ by a finite family. Bracketing
covers and $\delta$-covers provide precisely such discretizations. If
$\boldsymbol{x}\le\boldsymbol{y}$ coordinatewise, the parameter box
$[\boldsymbol{x},\boldsymbol{y}]$ is a $\delta$-bracket when
\[
  V(\boldsymbol{y})-V(\boldsymbol{x})\le\delta.
\]
A finite family of $\delta$-brackets covering $[0,1]^d$ is a
$\delta$-bracketing cover. Its minimum cardinality is denoted by
$N_{[]}(d,\delta)$.

A finite set $\Gamma\subset[0,1]^d$ is a $\delta$-cover if every
$\boldsymbol z\in[0,1]^d$ admits
$\boldsymbol x,\boldsymbol y\in\Gamma\cup\{\boldsymbol0\}$ such that
$\boldsymbol x\le\boldsymbol z\le\boldsymbol y$ and
$V(\boldsymbol y)-V(\boldsymbol x)\le\delta$. Let $N(d,\delta)$ be the
minimum cardinality of such a set. Monotonicity gives, for every
$\delta$-cover $\Gamma$, the standard estimate
\cite{DGS2005,Gnewuch2008JCo}
\begin{equation}
  D_n^*(P)
  \le
  \max_{\boldsymbol u\in\Gamma\cup\{\boldsymbol0\}}
  |\Delta_P(\boldsymbol u)|+\delta.
  \label{eq:qmc-reduction}
\end{equation}
Hence $N(d,\delta)$ measures the size of a finite discretization of the
defining supremum of the star discrepancy. The endpoints of a
$\delta$-bracketing cover form a $\delta$-cover, so
\[
  N(d,\delta)\le2N_{[]}(d,\delta).
\]
Cover cardinalities enter the union bounds and chaining estimates used
to prove bounds for $D^*(n,d)$ and $n^*(d,\varepsilon)$
\cite{Aistleitner2011,AistleitnerHofer2014,GnewuchHebbinghaus2021,GPW2021}.
Further applications to discrepancy computation and quasi-Monte Carlo
methods are reviewed in Section~\ref{sec:previous}.

\subsection{Main results}

We first establish finite lower bounds for both bracketing and $\delta$-cover
numbers. Combined with two explicit upper-bound constructions, these bounds
determine the sharp fixed-dimensional leading coefficient of
$N_{[]}(d,\delta)$ and give new fixed-dimensional bounds for $N(d,\delta)$.
We then use the resulting dimension-explicit finite estimates to improve the
known constant in the inverse star-discrepancy problem. Throughout,
$O_d(\cdot)$ and $o_d(\cdot)$ refer to the fixed-dimensional limit
$\delta\downarrow0$.

The finite lower bounds underlying both results are as follows.

\begin{theorem}
\label{thm:bracketing-lower}
\label{thm:cover-lower}
For every $d\in\mathbb N$ and $0<\delta\le1$,
\[
  N_{[]}(d,\delta)\ge\left\lceil\delta^{-d}\right\rceil,
  \qquad
  N(d,\delta)
  \ge
  \left\lceil \frac{d!}{d^d}\,\delta^{-d}\right\rceil.
\]
\end{theorem}

For every fixed $d\ge3$, the previously known bracketing bounds
\cite{Gnewuch2008JCo,Gnewuch2024} gave
\[
  1
  \le
  \liminf_{\delta\downarrow0}\delta^dN_{[]}(d,\delta)
  \le
  \limsup_{\delta\downarrow0}\delta^dN_{[]}(d,\delta)
  \le
  \frac{d^d}{d!}.
\]
In dimension two, coefficient one had already been attained from above
\cite{Gnewuch2008EJC}. Thus the optimal first-order coefficient remained
unknown in every fixed dimension $d\ge3$. We close this gap.

\begin{theorem}
\label{thm:main}
For every fixed $d\in\mathbb N$,
$N_{[]}(d,\delta)=(1+o_d(1))\delta^{-d}$ as $\delta\downarrow0$.
Equivalently,
\[
  \lim_{\delta\downarrow0}\delta^dN_{[]}(d,\delta)=1.
\]
\end{theorem}

For $d\ge2$, the sharp leading coefficient of the $\delta$-cover number
remains unknown.
Classical lower and upper bounds are given in
\cite{DGS2005,Gnewuch2008JCo}. The lower bound in
Theorem~\ref{thm:cover-lower} has the correct order, while the same
construction used for Theorem~\ref{thm:main} gives upper coefficient one.
Thus we obtain the following fixed-dimensional bounds.

\begin{theorem}
\label{thm:cover-main}
For every fixed $d\in\mathbb N$,
\[
  \frac{d!}{d^d}
  \le
  \liminf_{\delta\downarrow0}\delta^dN(d,\delta)
  \le
  \limsup_{\delta\downarrow0}\delta^dN(d,\delta)
  \le1.
\]
\end{theorem}

The proofs of Theorems~\ref{thm:main} and~\ref{thm:cover-main} are completed
in Section~\ref{sec:explicit}. For $d\ge2$, they follow immediately from
Theorem~\ref{thm:bracketing-lower} and Corollary~\ref{cor:closed-bounds},
while the case $d=1$ is elementary.

The second construction also improves the upper bound on exponential growth
in the dimension. For every fixed $0<\delta<1$,
Theorem~\ref{thm:bracketing-lower} and Corollary~\ref{cor:hs-upper-base} give
\[
  \delta^{-1}
  \le \liminf_{d\to\infty}N_{[]}(d,\delta)^{1/d}
  \le \limsup_{d\to\infty}N_{[]}(d,\delta)^{1/d}
  \le F(A_\delta),
\]
where $F$ and $A_\delta$ are defined in \eqref{eq:hs-A-F}. As
$\delta\downarrow0$, we have $F(A_\delta)=\delta^{-1}+e+O(\delta)$.
The ratio of this upper base to that of Gnewuch's refined bound
\eqref{eq:hs-gnewuch-refined} tends to $1/e$, with $d\to\infty$ taken first.

For the inverse star-discrepancy problem, Wei\ss{} \cite{Weiss2026}
proved the previously best explicit bounds
\[
  D^*(n,d)\le2.4631832\sqrt{\frac dn},
  \qquad
  n^*(d,\varepsilon)\le6.0672715\,d\varepsilon^{-2}.
\]
To improve these bounds, we derive explicit covering estimates in both
$d$ and $\delta$ from two complementary constructions and use them to prove
the following result.

\begin{theorem}
\label{thm:discrepancy-23463}
For every $d,n\in\mathbb N$,
\begin{equation}
  D^*(n,d)\le2.3463\sqrt{\frac dn}.
  \label{eq:discrepancy-23463}
\end{equation}
Consequently, for every $d\in\mathbb N$ and $0<\varepsilon<1$,
\begin{equation}
  n^*(d,\varepsilon)
  \le
  \left\lceil5.5052\,d\varepsilon^{-2}\right\rceil.
  \label{eq:inverse-55052}
\end{equation}
\end{theorem}

\subsection{Proof ideas and constructions}

We use the polynomial product measure of
Gnewuch, Wahlstr\"om, and Winzen \cite{GnewuchWahlstromWinzen2012}
\begin{equation}
  \mathrm d\pi_d(\boldsymbol z)
  =d^dV(\boldsymbol z)^{d-1}\,\mathrm d\boldsymbol z.
  \label{eq:intro-pi}
\end{equation}
We prove that every parameter bracket $[\boldsymbol x,\boldsymbol y]$
satisfies the sharp mass bound
\[
  \pi_d([\boldsymbol x,\boldsymbol y])
  \le
  \bigl(V(\boldsymbol y)-V(\boldsymbol x)\bigr)^d.
\]
Thus every $\delta$-bracket has $\pi_d$-measure at most $\delta^d$, which
gives the bracketing lower bound by subadditivity. A related simplex estimate
for the same measure gives the $\delta$-cover lower bound. For nondegenerate
brackets, equality holds precisely in the homothetic case. This geometry also
guides both upper-bound constructions. For the relation of $\pi_d$ and the
homothetic geometry to previous work, see Remarks~\ref{rem:measure-context}
and~\ref{rem:lower-equality}.

Our first construction uses locally adapted rectangular grids. Its leading
local counting density, after multiplication by $\delta^d$, is
$d^dV(\boldsymbol z)^{d-1}$, the density of $\pi_d$.
It therefore attains the sharp fixed-dimensional leading coefficient.
Shared vertices give upper coefficient one for the $\delta$-cover number,
and the grid counts yield simple finite bounds.

The second construction uses homothetic brackets. Under
$\xi_j=-d\log z_j$, these brackets become cubes and volume level sets become
hyperplanes. A thin-shell decomposition reduces the counting to integer
compositions and gives better control of the exponential growth in $d$
for fixed $\delta$. The same construction also attains the sharp
fixed-dimensional leading coefficient. In Section~\ref{sec:discrepancy},
we use the grid bounds in smaller dimensions and the shell bounds in larger
dimensions to estimate the coarse levels of the chaining argument.

The remainder of the paper is organized as follows.
Section~\ref{sec:previous} reviews previous bounds and discrepancy
applications. The lower bounds are proved in Section~\ref{sec:lower-bounds},
and the anisotropic construction, including its explicit bounds, is developed
in Section~\ref{sec:anisotropic-construction}.
Section~\ref{sec:homothetic-shells} develops the complementary homothetic
logarithmic-shell construction in both high-dimensional and fixed-dimensional
regimes. Finally, Section~\ref{sec:discrepancy} combines the two families of
bounds in the proof of Theorem~\ref{thm:discrepancy-23463}.

\section{Previous bounds}
\label{sec:previous}

For a broader survey of bracketing entropy, randomization,
derandomization, and their applications to discrepancy, see
\cite{Gnewuch2012Survey}.

\subsection{Bracketing numbers}
\label{sec:previous-bounds}

Gnewuch established general lower and upper bounds for the bracketing
numbers of anchored and unanchored axis-parallel boxes
\cite{Gnewuch2008JCo}. In particular, for every fixed $d\ge2$,
\[
  N_{[]}(d,\delta)
  \ge \delta^{-d}\bigl(1-O_d(\delta)\bigr)
  \qquad (\delta\downarrow0).
\]
Theorem~\ref{thm:bracketing-lower} strengthens this asymptotic estimate to
the exact finite bound
$N_{[]}(d,\delta)\ge\lceil\delta^{-d}\rceil$, valid for every
$d\in\mathbb N$ and $0<\delta\le1$. In dimension one,
$N_{[]}(1,\delta)=\lceil\delta^{-1}\rceil$.

The coefficient one was also attained from above in dimension two. Gnewuch
\cite{Gnewuch2008EJC} constructed re-oriented covers satisfying
\[
  N_{[]}(2,\delta)\le\delta^{-2}+o(\delta^{-2}).
\]
Gnewuch, Pasing, and Wei\ss{} used a generalized Faulhaber inequality to
prove the estimate
\begin{equation}
  N_{[]}(d,\delta)
  \le
  b_d\frac{d^d}{d!}(\delta^{-1}+1)^d,
  \qquad
  b_d=\max\{1,1.1^{d-101}\},
  \label{eq:gpw2021}
\end{equation}
for every $d\in\mathbb N$ and $0<\delta\le1$; see
\cite[Theorem~2.5]{GPW2021}. This dimension-explicit estimate was used in
their discrepancy applications.

Thi\'emard introduced a recursive bracketing construction in connection with
certified star-discrepancy computation \cite{Thiemard2001}. Gnewuch later
gave a refined analysis of this construction and obtained, for $d\ge3$,
\begin{equation}
  N_{[]}(d,\delta)
  \le
  \frac{d^d}{d!}\,\delta^{-d},
  \qquad 0<\delta<1.
  \label{eq:gnewuch2024}
\end{equation}
See \cite{Gnewuch2024}. This improves \eqref{eq:gpw2021} for $d\ge3$.

\subsection{\texorpdfstring{$\delta$}{delta}-cover numbers and approximation of star discrepancy}
\label{sec:previous-cover-bounds}

Doerr, Gnewuch, and Srivastav proved a lower bound of the correct order for
$N(d,\delta)$. More precisely, their Theorem~2.8 gives, for $d\ge2$ and
$0<\delta\le1/d$,
\begin{equation}
  N(d,\delta)\ge
  \frac25\frac{d!}{d^d}\,\delta^{-d}
  -\frac25d!\sum_{k=0}^{d-1}
  \frac{d^k\lvert\log(d\delta)\rvert^k}{k!}.
  \label{eq:dgs-cover-lower}
\end{equation}
See \cite[Theorem~2.8]{DGS2005}. Thus their asymptotic leading coefficient is
$\frac25d!/d^d$. Theorem~\ref{thm:cover-lower} improves this coefficient by
a factor $5/2$, removes the negative logarithmic correction, and extends
the bound to all $d\in\mathbb N$ and $0<\delta\le1$.

Equation~\eqref{eq:qmc-reduction} gives a certified additive approximation of
the star discrepancy from its values on a $\delta$-cover. This is the
application underlying Thi\'emard's algorithm \cite{Thiemard2001}.
The discretization
inequality also appears in \cite[Lemma~3.1]{DGS2005} and
\cite[(19)]{Gnewuch2008JCo}. The cardinality $N(d,\delta)$ measures the number
of test boxes, whereas the arithmetic cost of evaluating them also depends
on the range-counting procedure. Covers are also used in deterministic
constructions of low-discrepancy point sets; the dependent-randomized-rounding
algorithm of Doerr, Gnewuch, and Wahlstr\"om
\cite{DoerrGnewuchWahlstrom2010} improves the earlier derandomization approach
of \cite{DGS2005} and includes numerical experiments in dimensions up to 21.
Another application appears in Markov chain quasi-Monte Carlo, where Dick,
Rudolf, and Zhu use covering arguments in discrepancy estimates for uniformly
ergodic Markov chains \cite{DickRudolfZhu2016}.

\section{Lower bounds}
\label{sec:lower-bounds}

\subsection{Bracketing numbers}
\label{sec:lower}

\begin{proof}[Proof of the bracketing bound in Theorem~\ref{thm:bracketing-lower}]
The lower bound follows from a probability measure whose mass on a parameter
bracket is controlled by the $d$th power of its width. We use the polynomial
product measure $\pi_d$ from \cite{GnewuchWahlstromWinzen2012}, defined on
$[0,1]^d$ by
\begin{equation}
  \mathrm d\pi_d(\boldsymbol z)
  =d^d\left(\prod_{j=1}^d z_j\right)^{d-1}
  \mathrm d\boldsymbol z.
  \label{eq:lower-measure}
\end{equation}
This is a probability measure.

Let $B=[\boldsymbol x,\boldsymbol y]$. If some $y_j=0$, then
$\pi_d(B)=0$. Otherwise set $r_j=x_j/y_j$. By
\eqref{eq:lower-measure},
\[
  \pi_d(B)
  =\prod_{j=1}^d(y_j^d-x_j^d)
  =V(\boldsymbol y)^d\prod_{j=1}^d(1-r_j^d).
\]
Applying the arithmetic--geometric mean inequality twice gives
\[
  \left(\prod_{j=1}^d(1-r_j^d)\right)^{1/d}
  \le \frac1d\sum_{j=1}^d(1-r_j^d)
  =1-\frac1d\sum_{j=1}^d r_j^d
  \le 1-\prod_{j=1}^d r_j.
\]
Hence
\begin{equation}
  \pi_d(B)
  \le
  \bigl(V(\boldsymbol y)-V(\boldsymbol x)\bigr)^d.
  \label{eq:bracket-measure-bound}
\end{equation}
Every $\delta$-bracket therefore has $\pi_d$-measure at most $\delta^d$.
If a bracketing cover contains $M$ brackets, subadditivity yields
\[
  1\le M\delta^d.
\]
Thus $M\ge\delta^{-d}$. Since $M$ is an integer, the claim follows.
\end{proof}

\begin{remark}
\label{rem:measure-context}
The measure in \eqref{eq:lower-measure} is exactly the polynomial product
measure used by Gnewuch, Wahlstr\"om, and Winzen
\cite[p.~786 and Section~5.1]{GnewuchWahlstromWinzen2012} for nonuniform
sampling in their threshold-accepting algorithm for star-discrepancy
approximation. Its role here is different: \eqref{eq:bracket-measure-bound}
turns it into a covering measure whose mass is controlled sharply by bracket
width. Related volume-biased product measures also occur in generalized
discrepancy \cite{DickL1_2026,NovakPillichshammer2026}.
\end{remark}

\begin{remark}
\label{rem:lower-equality}
For a nondegenerate bracket, equality in
\eqref{eq:bracket-measure-bound} holds precisely when
$x_1/y_1=\cdots=x_d/y_d$. The extremal brackets are therefore homothetic.
This agrees with a different extremal property proved by Gnewuch
\cite[Lemma~1.1]{Gnewuch2008JCo}: among $\delta$-brackets with a fixed upper
right corner, a bracket of largest Lebesgue volume is homothetic. In the
lower-bound argument of \cite[Lemma~1.2]{Gnewuch2008JCo}, the varying
Lebesgue volumes of such maximal brackets lead to a continuous averaging
procedure. In contrast, \eqref{eq:bracket-measure-bound} gives the uniform
bound $\pi_d(B)\le\delta^d$ for every $\delta$-bracket, so subadditivity
reduces the lower bound to a discrete counting argument.
\end{remark}

\subsection{\texorpdfstring{$\delta$}{delta}-cover numbers}
\label{sec:cover-lower}

\begin{proof}[Proof of the cover bound in Theorem~\ref{thm:cover-lower}]
Let $\Gamma$ be a $\delta$-cover. For $\boldsymbol y\in\Gamma$, put
\[
  C_{\boldsymbol y}
  =\left\{
    \boldsymbol z\le\boldsymbol y:
    V(\boldsymbol y)-V(\boldsymbol z)\le\delta
  \right\}.
\]
For every $\boldsymbol z\ne\boldsymbol0$, choose
$\boldsymbol x,\boldsymbol y\in\Gamma\cup\{\boldsymbol0\}$ with
$\boldsymbol x\le\boldsymbol z\le\boldsymbol y$ and
$V(\boldsymbol y)-V(\boldsymbol x)\le\delta$. Then
$\boldsymbol y\in\Gamma$ and
$V(\boldsymbol y)-V(\boldsymbol z)\le\delta$, so
$\boldsymbol z\in C_{\boldsymbol y}$. Thus these sets cover
$[0,1]^d\setminus\{\boldsymbol0\}$. We estimate their masses using the same
probability measure $\pi_d$ defined in \eqref{eq:lower-measure}.
Let $\lambda_d$ denote $d$-dimensional Lebesgue measure.

Fix $\boldsymbol y\in\Gamma$. If $V(\boldsymbol y)\le\delta$, then
\[
  \pi_d(C_{\boldsymbol y})
  \le\pi_d([\boldsymbol0,\boldsymbol y])
  =V(\boldsymbol y)^d
  \le\delta^d
  \le\frac{d^d}{d!}\,\delta^d.
\]
Suppose that $V(\boldsymbol y)>\delta$ and set
\[
  u_j=\left(\frac{z_j}{y_j}\right)^d,
  \qquad j=1,\ldots,d.
\]
Since $\mathrm d u_j=d\,z_j^{d-1}y_j^{-d}\,\mathrm d z_j$, we have
\[
  \mathrm d\pi_d(\boldsymbol z)
  =V(\boldsymbol y)^d\,\mathrm d\boldsymbol u.
\]
Under this change of variables, $C_{\boldsymbol y}$ is mapped onto
\[
  \left\{
    \boldsymbol u\in[0,1]^d:
    \left(\prod_{j=1}^d u_j\right)^{1/d}
    \ge1-\frac{\delta}{V(\boldsymbol y)}
  \right\}.
\]
Hence
\[
  \pi_d(C_{\boldsymbol y})
  =V(\boldsymbol y)^d\lambda_d\left(
    \left\{
      \boldsymbol u\in[0,1]^d:
      \left(\prod_{j=1}^d u_j\right)^{1/d}
      \ge1-\frac{\delta}{V(\boldsymbol y)}
    \right\}
  \right).
\]
For every $\boldsymbol u$ in this set, the arithmetic--geometric mean
inequality gives
\[
  \frac1d\sum_{j=1}^d u_j
  \ge\left(\prod_{j=1}^d u_j\right)^{1/d}
  \ge1-\frac{\delta}{V(\boldsymbol y)},
\]
or equivalently
\[
  \sum_{j=1}^d(1-u_j)
  \le\frac{d\delta}{V(\boldsymbol y)}.
\]
Thus the transformed region is contained in a simplex of volume
$\bigl(d\delta/V(\boldsymbol y)\bigr)^d/d!$, and therefore
\[
  \pi_d(C_{\boldsymbol y})
  \le V(\boldsymbol y)^d
  \frac{1}{d!}
  \left(\frac{d\delta}{V(\boldsymbol y)}\right)^d
  =\frac{d^d}{d!}\,\delta^d.
\]
If $M=|\Gamma|$, subadditivity now yields
\[
  1
  \le\sum_{\boldsymbol y\in\Gamma}\pi_d(C_{\boldsymbol y})
  \le M\frac{d^d}{d!}\,\delta^d.
\]
Therefore $M\ge(d!/d^d)\delta^{-d}$, and the integer-valued statement
follows.
\end{proof}

\begin{remark}
\label{rem:cover-lower-comparison}
The standard relation \cite[(1)]{Gnewuch2008JCo}
\[
  2N_{[]}(d,\delta)
  \le N(d,\delta)\bigl(N(d,\delta)+1\bigr)
\]
would give only a lower bound of order $\delta^{-d/2}$ when combined with
the bracketing bound in Theorem~\ref{thm:bracketing-lower}. The cover bound in
Theorem~\ref{thm:cover-lower} obtains the correct exponent directly.
\end{remark}

\section{The anisotropic construction}
\label{sec:anisotropic-construction}

\subsection{Construction and bracketing bounds}
\label{sec:construction}

Fix $d\ge2$, $0<\delta<1$, and $K\in\mathbb N$. Put
\[
  \ell=\frac{1-\delta}{K},
  \qquad
  a_r=\delta+r\ell
  \quad (r=0,\ldots,K).
\]
Partition $[\delta,1]^d$ into the coarse boxes
\[
  Q_{\boldsymbol k}
  =\prod_{j=1}^d[a_{k_j-1},a_{k_j}],
  \qquad
  \boldsymbol k\in\{1,\ldots,K\}^d,
\]
and let
$\boldsymbol q_{\boldsymbol k}=(a_{k_1},\ldots,a_{k_d})$ be the upper
endpoint of $Q_{\boldsymbol k}$. Define
\begin{align}
  A_j(\boldsymbol q)
  &:=\prod_{i\ne j}q_i, \notag\\
  m_j(\boldsymbol k)
  &:=\left\lceil
    \frac{d\ell A_j(\boldsymbol q_{\boldsymbol k})}{\delta}
  \right\rceil.
  \label{eq:local-definitions}
\end{align}
The $j$th side of $Q_{\boldsymbol k}$ is divided into
$m_j(\boldsymbol k)$ equal intervals. The resulting microboxes are the
interior brackets.

The boundary region is handled separately. For each
$j\in\{1,\ldots,d\}$, let $\boldsymbol b^{(j)}$ have $j$th coordinate
$\delta$ and all other coordinates equal to one. Add the bracket
$[\boldsymbol0,\boldsymbol b^{(j)}]$.

\begin{figure}[H]
\centering
\begin{tikzpicture}[line cap=round,line join=round]
  \begin{scope}[x=6.6cm,y=6.6cm]
    \fill[gray!18] (0,0) rectangle (0.05,1);
    \fill[gray!18] (0,0) rectangle (1,0.05);
    \draw[black!75,line width=0.55pt] (0,0) rectangle (1,1);
    \foreach \xl/\xu/\my in {
      0.05/0.185714/2,
      0.185714/0.321429/2,
      0.321429/0.457143/3,
      0.457143/0.592857/4,
      0.592857/0.728571/4,
      0.728571/0.864286/5,
      0.864286/1/6}{
      \foreach \yl/\yu/\mx in {
        0.05/0.185714/2,
        0.185714/0.321429/2,
        0.321429/0.457143/3,
        0.457143/0.592857/4,
        0.592857/0.728571/4,
        0.728571/0.864286/5,
        0.864286/1/6}{
        \pgfmathtruncatemacro{\mxm}{\mx-1}
        \pgfmathtruncatemacro{\mym}{\my-1}
        \ifnum\mxm>0
          \foreach \i in {1,...,\mxm}{
            \pgfmathsetmacro{\xx}{\xl+\i*(\xu-\xl)/\mx}
            \draw[blue!48,line width=0.18pt] (\xx,\yl)--(\xx,\yu);
          }
        \fi
        \ifnum\mym>0
          \foreach \j in {1,...,\mym}{
            \pgfmathsetmacro{\yy}{\yl+\j*(\yu-\yl)/\my}
            \draw[blue!48,line width=0.18pt] (\xl,\yy)--(\xu,\yy);
          }
        \fi
      }
    }
    \foreach \t in {0.05,0.185714,0.321429,0.457143,0.592857,0.728571,0.864286}{
      \draw[black!65,line width=0.42pt] (\t,0.05)--(\t,1);
      \draw[black!65,line width=0.42pt] (0.05,\t)--(1,\t);
    }
    \draw[black!85,line width=0.95pt]
      (0.592857,0.321429) rectangle (0.728571,0.457143);
    \node[below left,font=\scriptsize] at (0,0) {$0$};
    \node[below,font=\scriptsize] at (0.05,0) {$\delta$};
    \node[left,font=\scriptsize] at (0,0.05) {$\delta$};
    \node[below,font=\scriptsize] at (1,0) {$1$};
    \node[left,font=\scriptsize] at (0,1) {$1$};
    \node[rotate=90,font=\scriptsize,gray!65!black] at (0.025,0.54)
      {$[\boldsymbol{0},\boldsymbol{b}^{(1)}]$};
    \node[font=\scriptsize,gray!65!black] at (0.54,0.025)
      {$[\boldsymbol{0},\boldsymbol{b}^{(2)}]$};
    \node[font=\small] at (0.50,1.075)
      {$d=2,\ \delta=0.05,\ K=7$};
  \end{scope}
  \draw[-{Latex[length=2.2mm]},black!75,line width=0.6pt]
    (4.75cm,2.62cm) .. controls (5.85cm,2.45cm) and (6.35cm,2.25cm) .. (7.15cm,2.18cm);
  \begin{scope}[shift={(7.45cm,0.95cm)},x=4.5cm,y=4.5cm]
    \draw[black!85,line width=0.9pt] (0,0) rectangle (1,1);
    \foreach \i in {1,2}
      \draw[blue!60,line width=0.36pt] ({\i/3},0)--({\i/3},1);
    \foreach \j in {1,2,3}
      \draw[blue!60,line width=0.36pt] (0,{\j/4})--(1,{\j/4});
    \fill[blue!15] (2/3,3/4) rectangle (1,1);
    \draw[blue!70!black,line width=0.7pt] (2/3,3/4) rectangle (1,1);
    \fill[blue!70!black] (2/3,3/4) circle[radius=0.012];
    \fill[blue!70!black] (1,1) circle[radius=0.012];
    \node[below left,font=\scriptsize] at (2/3,3/4) {$\boldsymbol{x}$};
    \node[above right,font=\scriptsize] at (1,1) {$\boldsymbol{y}$};
    \node[font=\scriptsize,align=center] at (0.5,-0.15)
      {a representative coarse box $Q_{\boldsymbol{k}}$};
    \node[font=\scriptsize,align=center] at (0.5,-0.29)
      {$m_1(\boldsymbol{k})=3$ and $m_2(\boldsymbol{k})=4$};
  \end{scope}
\end{tikzpicture}
\caption{The construction for $d=2$, $\delta=0.05$, and $K=7$.
The shaded boundary strips are covered separately. The remaining square is
partitioned into coarse boxes, each carrying a box-dependent anisotropic
grid. The enlargement shows one representative coarse box. Every microbox
is used as a bracket. The TikZ code for this explanatory figure was
prepared with the assistance of ChatGPT 5.6 Sol (OpenAI).}
\label{fig:construction-d2}
\end{figure}
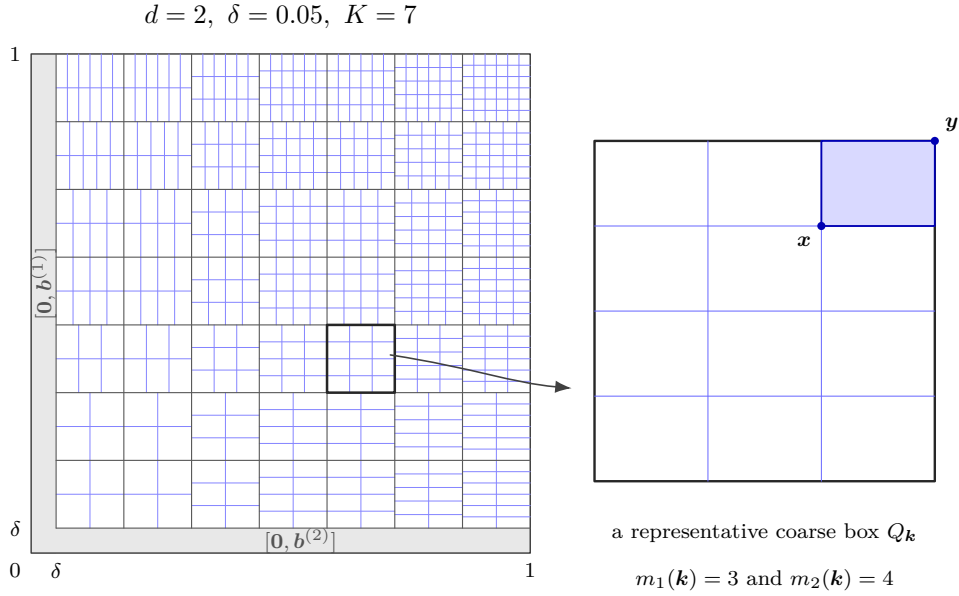

Figure~\ref{fig:construction-d2} shows the construction in dimension two.

Write
\[
  T_{d-1}=T_{d-1}(K,\delta)=\sum_{r=1}^K a_r^{d-1}.
\]
We use the following counting inequality for both brackets and vertices.
For $b,c\ge0$ and $q_1,\ldots,q_d\ge0$, H\"older's inequality
for a sum of two terms gives, for each $j$,
\[
  b+c\prod_{i\ne j}q_i
  \le
  \prod_{i\ne j}
  \left(b+cq_i^{d-1}\right)^{1/(d-1)}.
\]
For $d=2$, this is an equality. Multiplying over $j$ yields
\[
  \prod_{j=1}^d\left(b+cA_j(\boldsymbol q)\right)
  \le
  \prod_{j=1}^d\left(b+cq_j^{d-1}\right).
\]
Consequently, the Cartesian structure of the coarse grid gives
\begin{equation}
  \sum_{\boldsymbol k}
  \prod_{j=1}^d\left(b+cA_j(\boldsymbol q_{\boldsymbol k})\right)
  \le
  \sum_{\boldsymbol k}\prod_{j=1}^d
  \left(b+ca_{k_j}^{d-1}\right)
  =\left(bK+cT_{d-1}\right)^d.
  \label{eq:grid-counting}
\end{equation}

\begin{theorem}
\label{thm:parameterized}
For every $d\ge2$, $0<\delta<1$, and $K\in\mathbb N$, the construction above
is a $\delta$-bracketing cover. If $M_{d,\delta,K}$ denotes its number of
brackets, then
\begin{equation}
  N_{[]}(d,\delta)\le M_{d,\delta,K}
  \le d+\delta^{-d}
  \left(d\ell T_{d-1}+\delta K\right)^d.
  \label{eq:bracketing-bound}
\end{equation}
\end{theorem}

\begin{proof}
We first verify coverage. If $z_j\le\delta$ for some $j$, then
$\boldsymbol0\le\boldsymbol z\le\boldsymbol b^{(j)}$ and
$V(\boldsymbol b^{(j)})=\delta$. Otherwise
$\boldsymbol z\in[\delta,1]^d$ and belongs to a microbox of one of the
coarse boxes.

Let $[\boldsymbol x,\boldsymbol y]$ be a microbox in
$Q_{\boldsymbol k}$ and write $\boldsymbol q=\boldsymbol q_{\boldsymbol k}$.
By \eqref{eq:local-definitions}, $y_j-x_j\le
\delta/(dA_j(\boldsymbol q))$. The exact telescoping identity and
$\boldsymbol x,\boldsymbol y\le\boldsymbol q$ give
\[
\begin{aligned}
  V(\boldsymbol y)-V(\boldsymbol x)
  &=\sum_{j=1}^d
  (y_j-x_j)
  \left(\prod_{i<j}y_i\right)
  \left(\prod_{i>j}x_i\right)\\
  &\le\sum_{j=1}^d (y_j-x_j)A_j(\boldsymbol q)
  \le\delta.
\end{aligned}
\]
Thus every microbox is a $\delta$-bracket.

Since $m_j(\boldsymbol k)\le
1+(d\ell/\delta)A_j(\boldsymbol q_{\boldsymbol k})$,
\eqref{eq:grid-counting} with $b=1$ and $c=d\ell/\delta$ gives
\[
  \sum_{\boldsymbol k}\prod_{j=1}^d m_j(\boldsymbol k)
  \le
  \left(K+\frac{d\ell T_{d-1}}{\delta}\right)^d.
\]
Adding the $d$ boundary brackets proves \eqref{eq:bracketing-bound}.
\end{proof}

\subsection{The induced \texorpdfstring{$\delta$}{delta}-cover}
\label{sec:cover}

\begin{proposition}
\label{prop:cover-count}
Let $\Gamma_{d,\delta,K}$ consist of all vertices of all local microgrids.
Then $\Gamma_{d,\delta,K}$ is a $\delta$-cover and
\begin{equation}
  N(d,\delta)\le |\Gamma_{d,\delta,K}|
  \le \delta^{-d}
  \left(d\ell T_{d-1}+2\delta K\right)^d.
  \label{eq:cover-bound}
\end{equation}
\end{proposition}

\begin{proof}
Since $a_0=\delta$ and $a_K=1$, each $\boldsymbol b^{(j)}$ is a
coarse-grid vertex and hence a local-grid vertex. By
Theorem~\ref{thm:parameterized}, every point of $[\delta,1]^d$ lies in a
microbox whose lower and upper vertices belong to
$\Gamma_{d,\delta,K}$ and whose volume difference is at most $\delta$.
If $z_j\le\delta$ for some $j$, then
$\boldsymbol0\le\boldsymbol z\le\boldsymbol b^{(j)}$ and
$V(\boldsymbol b^{(j)})=\delta$. Hence $\Gamma_{d,\delta,K}$ is a
$\delta$-cover.

Summing the local vertex counts and using
$m_j(\boldsymbol k)+1\le
2+(d\ell/\delta)A_j(\boldsymbol q_{\boldsymbol k})$,
\eqref{eq:grid-counting} with $b=2$ and $c=d\ell/\delta$ gives
\[
  |\Gamma_{d,\delta,K}|
  \le
  \sum_{\boldsymbol k}\prod_{j=1}^d(m_j(\boldsymbol k)+1)
  \le
  \left(2K+\frac{d\ell T_{d-1}}{\delta}\right)^d.
\]
This proves \eqref{eq:cover-bound}.
\end{proof}

\begin{remark}
The estimate in Proposition~\ref{prop:cover-count} is not obtained by
counting the two endpoints of every bracket in
Theorem~\ref{thm:parameterized}. That argument would give only
\[
  N(d,\delta)\le 2M_{d,\delta,K}.
\]
Instead, $\Gamma_{d,\delta,K}$ contains each local-grid vertex only once.
Adjacent microboxes share most of their vertices, and counting the distinct
vertices removes the factor $2$ from the leading term. At the asymptotic
level, this is what gives the upper coefficient $1$ in
Theorem~\ref{thm:cover-main}, rather than the coefficient $2$ obtained from
the general inequality $N(d,\delta)\le2N_{[]}(d,\delta)$.
\end{remark}

\subsection{Explicit bounds}
\label{sec:explicit}

The following elementary estimate sharpens the right-Riemann-sum error.

\begin{lemma}
\label{lem:right-sum}
Let $p\ge1$ be an integer, $h=(1-\delta)/K$, and
$a_r=\delta+rh$. Then
\begin{equation}
  h\sum_{r=1}^K a_r^p
  \le
  \frac{1-\delta^{p+1}}{p+1}
  +\frac h2(1-\delta^p)
  +\frac{ph^2}{12}(1-\delta^{p-1}).
  \label{eq:right-sum-quadratic}
\end{equation}
\end{lemma}

\begin{proof}
For $x\ge0$, define the local trapezoidal-rule error by
\[
  E_p(x,h)
  :=
  \frac h2\bigl(x^p+(x+h)^p\bigr)
  -\int_x^{x+h}t^p\,\mathrm dt.
\]
Expanding both sides gives
\begin{align*}
  E_p(x,h)
  &=
  \sum_{k=2}^p
  \binom pk\frac{k-1}{2(k+1)}x^{p-k}h^{k+1},\\
  \frac{ph^2}{12}\bigl((x+h)^{p-1}-x^{p-1}\bigr)
  &=
  \sum_{k=2}^p
  \binom pk\frac{k}{12}x^{p-k}h^{k+1}.
\end{align*}
Here the terms with $k=0,1$ in the first expansion vanish. Moreover,
\[
  \frac{k-1}{2(k+1)}\le\frac{k}{12}
  \qquad (2\le k\le p),
\]
since this is equivalent to $(k-2)(k-3)\ge0$. Hence
\[
  E_p(x,h)
  \le
  \frac{ph^2}{12}\bigl((x+h)^{p-1}-x^{p-1}\bigr).
\]

Summing this local inequality with $x=a_{r-1}$ for
$r=1,\ldots,K$ and using telescoping gives
\[
  \frac h2\left(\delta^p+2\sum_{r=1}^{K-1}a_r^p+1\right)
  -\int_\delta^1 t^p\,\mathrm dt
  \le
  \frac{ph^2}{12}(1-\delta^{p-1}).
\]
This implies \eqref{eq:right-sum-quadratic}.
\end{proof}

\begin{corollary}
\label{cor:closed-bounds}
Let $d\ge2$ and $0<\delta<1$. Define
\[
  A_{d,\delta}=d(1-\delta)(1-\delta^{d-1}),
  \qquad
  R_{d,\delta}
  =\frac{d-1}{6}
  \frac{(1-\delta)(1-\delta^{d-2})}{1-\delta^{d-1}}.
\]
Then
\begin{align}
  N_{[]}(d,\delta)
  &\le
  d+\delta^{-d}
  \left(
    1-\delta^d
    +\sqrt{2A_{d,\delta}\delta}
    +\delta(1+R_{d,\delta})
  \right)^d,
  \label{eq:closed-bracketing}\\
  N(d,\delta)
  &\le
  \delta^{-d}
  \left(
    1-\delta^d
    +2\sqrt{A_{d,\delta}\delta}
    +2\delta(1+R_{d,\delta})
  \right)^d.
  \label{eq:closed-cover}
\end{align}
For fixed $d$, both right-hand sides are
$\delta^{-d}\bigl(1+O_d(\sqrt\delta)\bigr)$ as $\delta\downarrow0$.
\end{corollary}

\begin{proof}
Lemma~\ref{lem:right-sum} with $p=d-1$ and $h=\ell$ gives
\[
  d\ell T_{d-1}
  \le
  1-\delta^d+\frac{A_{d,\delta}}{2K}
  +\frac{A_{d,\delta}R_{d,\delta}}{2K^2}.
\]
For \eqref{eq:closed-bracketing}, choose
\[
  K=\left\lceil
  \sqrt{\frac{A_{d,\delta}}{2\delta}}
  \right\rceil.
\]
Then
\[
  \frac{A_{d,\delta}}{2K}
  \le \sqrt{\frac{A_{d,\delta}\delta}{2}},
  \qquad
  \delta K
  \le \sqrt{\frac{A_{d,\delta}\delta}{2}}+\delta,
  \qquad
  \frac{A_{d,\delta}R_{d,\delta}}{2K^2}
  \le \delta R_{d,\delta}.
\]
Substitution into \eqref{eq:bracketing-bound} proves
\eqref{eq:closed-bracketing}.

For the covering estimate, choose
$K=\left\lceil\sqrt{A_{d,\delta}/(4\delta)}\right\rceil$.
Combining the same estimates as above with \eqref{eq:cover-bound} proves
\eqref{eq:closed-cover}. Finally,
$A_{d,\delta}=d+O_d(\delta)$ and
$R_{d,\delta}=O_d(1)$ give the stated fixed-dimensional estimate.
\end{proof}

\begin{corollary}
\label{cor:simple-bounds}
Let $d\ge2$ and $0<\delta<1$. Then
\begin{align}
  N_{[]}(d,\delta)
  &\le
  d+\delta^{-d}
  \left(1+\sqrt{2d\delta}+\frac{d+5}{6}\delta\right)^d
  \label{eq:simple-bracketing}\\
  &\le
  d+\delta^{-d}\exp\left(d\sqrt{2d\delta}\right),
  \label{eq:exponential-bracketing}\\
  N(d,\delta)
  &\le
  \delta^{-d}
  \left(1+2\sqrt{d\delta}+\frac{d+5}{3}\delta\right)^d
  \label{eq:simple-cover}\\
  &\le
  \delta^{-d}\exp\left(2d\sqrt{d\delta}\right).
  \label{eq:exponential-cover}
\end{align}
\end{corollary}

\begin{proof}
The polynomial bounds follow from Corollary~\ref{cor:closed-bounds} by using
$1-\delta^d\le1$, $A_{d,\delta}\le d$, and
$R_{d,\delta}\le(d-1)/6$.
For the exponential bounds, note that $d\ge2$ implies
$(d+5)/(12d)\le1/2$, so for every $u\ge0$,
\[
  1+u+\frac{d+5}{12d}u^2
  \le 1+u+\frac{u^2}{2}
  \le e^u.
\]
Apply this with $u=\sqrt{2d\delta}$ and $u=2\sqrt{d\delta}$, respectively,
and raise the inequalities to the $d$th power.
\end{proof}

Together with the lower bounds in Theorem~\ref{thm:bracketing-lower} and the
elementary one-dimensional case, Corollary~\ref{cor:closed-bounds}
immediately proves Theorems~\ref{thm:main} and~\ref{thm:cover-main}.

\section{The homothetic logarithmic-shell construction}
\label{sec:homothetic-shells}

The anisotropic construction in Section~\ref{sec:anisotropic-construction}
gives particularly simple explicit finite bounds. We now give a second
construction based on the homothetic geometry singled out by
Remark~\ref{rem:lower-equality}. It is especially effective in high
dimension and also recovers the sharp fixed-dimensional bracketing leading
constant.

\subsection{Construction and exact counts}

Throughout this section, let $0<\delta<1$. For
$\boldsymbol z\in(0,1]^d$, consider the logarithmic change of variables
\[
  \xi_j=-d\log z_j,
  \qquad j=1,\ldots,d.
\]
This change of variables sends each constant-volume hypersurface to a
hyperplane, since
\begin{equation}
  \frac1d\sum_{j=1}^d \xi_j=-\log V(\boldsymbol z).
  \label{eq:hs-log-volume}
\end{equation}
It also sends every homothetic bracket to an axis-parallel cube. Indeed, if
$[\boldsymbol x,\boldsymbol y]$ is homothetic with $x_j=ry_j$ for all $j$,
then every transformed edge has the same length,
\[
  (-d\log x_j)-(-d\log y_j)=-d\log r.
\]

The construction applies this correspondence shell by shell. For suitable
levels $L_i<L_{i+1}$, the region between two volume level sets,
\[
  e^{-L_{i+1}}<V(\boldsymbol z)\le e^{-L_i},
\]
is transformed into the diagonal strip
\begin{equation}\label{eq:ith-layer}
  dL_i\le\sum_{j=1}^d\xi_j<dL_{i+1}.
\end{equation}
We cover the $i$th strip by axis-parallel cubes of side length $c_i$, chosen
so that their inverse images are $\delta$-brackets. Thus the curved shell
geometry in $\boldsymbol z$-space is replaced by a uniform cubic grid in
$\boldsymbol\xi$-space. Figure~\ref{fig:homothetic-log-shell} illustrates
this basic correspondence in dimension two. The precise choices of the shell
levels, cube side lengths, and selected cubes are given below.

Define recursively
\begin{equation}
  L_0=0,
  \qquad
  c_i:=\log\bigl(1+\delta e^{L_i}\bigr),
  \qquad
  L_{i+1}:=\min\left\{L_i+\frac{c_i}{\sqrt{d\delta}},
  d\log(\delta^{-1})\right\}.
  \label{eq:hs-shell-recursion}
\end{equation}
Let $J$ be the first index for which
$L_J=d\log(\delta^{-1})$. Since $c_i\ge\log(1+\delta)$, such $J$ exists with
\begin{equation}
  J\le 1+
  \frac{d^{3/2}\sqrt\delta\,\log(\delta^{-1})}{\log(1+\delta)}.
  \label{eq:hs-J-crude}
\end{equation}
For a point $\boldsymbol{\xi}=(\xi_1,\dots,\xi_d)$ in the $i$th strip,
\eqref{eq:ith-layer} implies
\[
  d\left(\frac{L_i}{c_i}-1\right)
  \le \sum_{j=1}^d \left\lfloor\frac{\xi_j}{c_i}\right\rfloor
  <\frac{dL_{i+1}}{c_i}.
\]
Thus it suffices to select the cubes
$\prod_{j=1}^d[c_in_j,c_i(n_j+1)]$, whose indices
$\boldsymbol n=(n_1,\dots,n_d)\in\mathbb N_0^d$ have total index
$|\boldsymbol n|:=\sum_{j=1}^d n_j$ within these bounds.

Accordingly, define
\begin{equation}
  r_i^-:=\max\left\{0,
  \left\lceil d\left(\frac{L_i}{c_i}-1\right)\right\rceil
  \right\},
  \qquad
  r_i^+:=\left\lceil\frac{dL_{i+1}}{c_i}\right\rceil-1.
  \label{eq:hs-r-pm}
\end{equation}
With
\[
  \mathcal A_i
  :=\left\{
    \boldsymbol n\in\mathbb N_0^d:
    r_i^-\le|\boldsymbol n|\le r_i^+
  \right\},
\]
the number of nonnegative integer vectors $\boldsymbol n$ with
$|\boldsymbol n|=r$ is $\binom{r+d-1}{d-1}$. Hence the number of cubes
selected in the $i$th shell is exactly
\begin{equation}
  |\mathcal A_i|
  =\sum_{r=r_i^-}^{r_i^+}\binom{r+d-1}{d-1}
  =\binom{r_i^++d}{d}-\binom{r_i^-+d-1}{d},
  \label{eq:hs-A-count}
\end{equation}
where $\binom nd=0$ for integers $n<d$.

For each $\boldsymbol n\in\mathcal A_i$, the corresponding bracket in $\boldsymbol z$-space has endpoints
\begin{equation}
  x_{i,\boldsymbol n,j}
  :=\exp\left(-\frac{c_i(n_j+1)}d\right),
  \qquad
  y_{i,\boldsymbol n,j}
  :=\exp\left(-\frac{c_in_j}d\right).
  \label{eq:hs-xy}
\end{equation}
For $j=1,\ldots,d$, let $\boldsymbol b^{(j)}$ have $j$th coordinate
$\delta$ and all other coordinates equal to one. The logarithmic-shell
bracketing family is
\[
  \mathcal B_{d,\delta}^{\rm sh}
  :=
  \left\{[\boldsymbol0,\boldsymbol b^{(j)}]:1\le j\le d\right\}
  \cup
  \left\{
    [\boldsymbol x_{i,\boldsymbol n},\boldsymbol y_{i,\boldsymbol n}]:
    0\le i<J,\ \boldsymbol n\in\mathcal A_i
  \right\}.
\]
The corresponding endpoint set is
\[
  \Gamma_{d,\delta}^{\rm sh}
  :=
  \left\{\boldsymbol b^{(j)}:1\le j\le d\right\}
  \cup
  \left\{
    \boldsymbol x_{i,\boldsymbol n},\boldsymbol y_{i,\boldsymbol n}:
    0\le i<J,\ \boldsymbol n\in\mathcal A_i
  \right\}.
\]

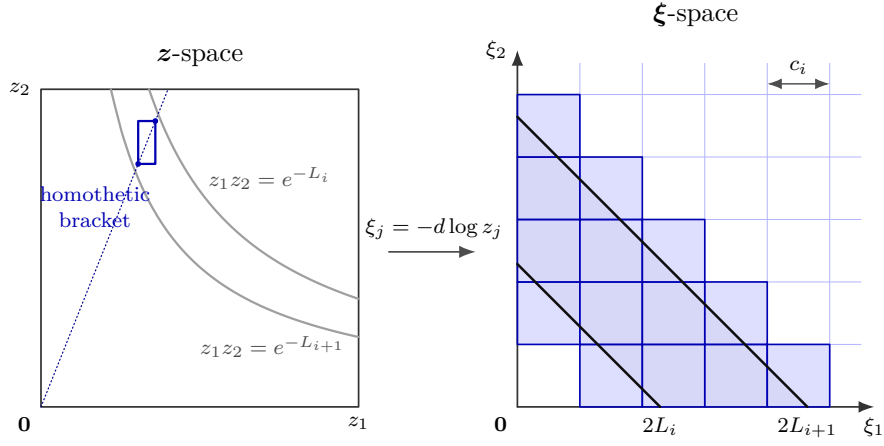
\begin{figure}[t]
\centering
\begin{tikzpicture}[line cap=round,line join=round,>=Latex]
  \begin{scope}[x=4.2cm,y=4.2cm]
    \draw[black!80,line width=0.65pt] (0,0) rectangle (1,1);
    \node[below left,font=\scriptsize] at (0,0) {$\boldsymbol0$};
    \node[below,font=\scriptsize] at (1,0) {$z_1$};
    \node[left,font=\scriptsize] at (0,1) {$z_2$};
    \draw[gray!75,line width=0.85pt,domain=0.22:1,smooth,variable=\x]
      plot ({\x},{0.22/\x});
    \draw[gray!75,line width=0.85pt,domain=0.34:1,smooth,variable=\x]
      plot ({\x},{0.34/\x});
    \node[black!70,font=\scriptsize] at (0.73,0.19)
      {$z_1z_2=e^{-L_{i+1}}$};
    \node[black!70,font=\scriptsize] at (0.72,0.72)
      {$z_1z_2=e^{-L_i}$};
    \draw[densely dotted,blue!55!black,line width=0.45pt]
      (0,0)--(0.40,1);
    \draw[blue!70!black,line width=0.85pt]
      (0.306,0.765) rectangle (0.36,0.90);
    \fill[blue!70!black] (0.306,0.765) circle[radius=0.009];
    \fill[blue!70!black] (0.36,0.90) circle[radius=0.009];
    \node[blue!70!black,font=\scriptsize,align=center] at (0.17,0.63)
      {homothetic\\bracket};
    \node[font=\small] at (0.5,1.10) {$\boldsymbol z$-space};
  \end{scope}

  \draw[-{Latex[length=2.2mm]},black!70,line width=0.65pt]
    (4.60,2.05)--(5.75,2.05);
  \node[font=\scriptsize] at (5.20,2.36) {$\xi_j=-d\log z_j$};

  \begin{scope}[shift={(6.30,0)},x=1.18cm,y=1.18cm]
    \def\A{1.60}
    \def\B{3.25}
    \def\Smax{4.00}
    \draw[->,black!80,line width=0.65pt] (0,0)--(\Smax,0);
    \draw[->,black!80,line width=0.65pt] (0,0)--(0,\Smax);
    \node[below left,font=\scriptsize] at (0,0) {$\boldsymbol0$};
    \node[below,font=\scriptsize] at (\Smax,0) {$\xi_1$};
    \node[left,font=\scriptsize] at (0,\Smax) {$\xi_2$};

    \fill[gray!13]
      (0,\A)--(0,\B)--(\B,0)--(\A,0)--cycle;
    \foreach \t in {0.70,1.40,2.10,2.80,3.50}{
      \draw[blue!30,line width=0.30pt] (\t,0)--(\t,3.85);
      \draw[blue!30,line width=0.30pt] (0,\t)--(3.85,\t);
    }
    \foreach \a/\b in {
      0/0.70,0.70/0,
      0/1.40,0.70/0.70,1.40/0,
      0/2.10,0.70/1.40,1.40/0.70,2.10/0,
      0/2.80,0.70/2.10,1.40/1.40,2.10/0.70,2.80/0}{
      \fill[blue!22,fill opacity=0.58] (\a,\b) rectangle ++(0.70,0.70);
      \draw[blue!70!black,line width=0.62pt]
        (\a,\b) rectangle ++(0.70,0.70);
    }
    \draw[black!95,line width=0.95pt] (0,\A)--(\A,0);
    \draw[black!95,line width=0.95pt] (0,\B)--(\B,0);
    \node[below,font=\scriptsize] at (\A,0) {$2L_i$};
    \node[below,font=\scriptsize] at (\B,0) {$2L_{i+1}$};
    \draw[<->,black!70,line width=0.55pt]
      (2.80,3.62)--(3.50,3.62);
    \node[above,font=\scriptsize] at (3.15,3.62) {$c_i$};
    \node[font=\small] at (2.00,4.40) {$\boldsymbol\xi$-space};
  \end{scope}
\end{tikzpicture}
\caption{The logarithmic-shell construction in dimension two. In
$\boldsymbol z$-space, the shell is bounded by product-volume level sets and a
homothetic bracket is shown in blue. Under $\xi_j=-d\log z_j$, the shell
becomes the diagonal strip $2L_i\le\xi_1+\xi_2<2L_{i+1}$, which is covered by
full squares of side length $c_i$. The TikZ code for this explanatory figure
was prepared with the assistance of ChatGPT 5.6 Sol (OpenAI).}
\label{fig:homothetic-log-shell}
\end{figure}

\begin{theorem}
\label{thm:hs-finite}
For every $d\in\mathbb N$ and $0<\delta<1$,
$\mathcal B_{d,\delta}^{\rm sh}$ is a $\delta$-bracketing cover and
$\Gamma_{d,\delta}^{\rm sh}$ is a $\delta$-cover.
\end{theorem}

\begin{proof}
The boundary brackets $[\boldsymbol0,\boldsymbol b^{(j)}]$ cover every point
having at least one coordinate at most $\delta$. Now let
$\boldsymbol z\in(\delta,1]^d$. Choose its shell $i$ and the index
$\boldsymbol n\in\mathcal A_i$ of the selected cube containing its transformed
point. Then
$\boldsymbol x_{i,\boldsymbol n}<\boldsymbol z\le
\boldsymbol y_{i,\boldsymbol n}$. Its width satisfies
\begin{align*}
  V(\boldsymbol y_{i,\boldsymbol n})
  -V(\boldsymbol x_{i,\boldsymbol n})
  &=
  \exp\left(-\frac{c_i|\boldsymbol n|}{d}\right)
  (1-e^{-c_i})\\
  &\le e^{-L_i+c_i}(1-e^{-c_i})\\
  &=e^{-L_i}(e^{c_i}-1)=\delta,
\end{align*}
where the inequality follows from the definition of $r_i^-$. Thus
$\mathcal B_{d,\delta}^{\rm sh}$ is a $\delta$-bracketing cover. Its
endpoints form the $\delta$-cover $\Gamma_{d,\delta}^{\rm sh}$.
\end{proof}

Consequently, the two constructions satisfy
\begin{align}
  N_{[]}(d,\delta)
  &\le |\mathcal B_{d,\delta}^{\rm sh}|
  \le d+\sum_{i=0}^{J-1}|\mathcal A_i|
  =:B_{[]}^{\rm sh}(d,\delta),
  \label{eq:hs-layer-finite}\\
  N(d,\delta)
  &\le |\Gamma_{d,\delta}^{\rm sh}|
  \le d+\sum_{i=0}^{J-1}
  |\mathcal A_i\cup(\mathcal A_i+\boldsymbol1)|
  =:B_{\rm cov}^{\rm sh}(d,\delta),
  \label{eq:hs-layer-finite-cover}
\end{align}
where $\boldsymbol1=(1,\ldots,1)$. The second sum counts distinct endpoints
within each shell; an endpoint shared by two different shells may be counted
twice. Since translation preserves cardinality,
\[
  B_{[]}^{\rm sh}(d,\delta)
  \le B_{\rm cov}^{\rm sh}(d,\delta)
  \le 2B_{[]}^{\rm sh}(d,\delta).
\]

\subsection{High-dimensional bounds}

The recursion also gives the following useful bound for the number of shells.

\begin{lemma}
\label{lem:hs-J-simple}
The shell number $J$ satisfies
\begin{equation}
  J\le
  1+\frac{\sqrt{d\delta}}{\log(1+\delta)}
  +(1+\sqrt{d\delta})
    \log\frac{d}{\delta\log(1+\delta)}.
  \label{eq:hs-J-simple}
\end{equation}
\end{lemma}

\begin{proof}
Put $\lambda=\log(\delta^{-1})$ and
$c(L)=\log(1+\delta e^L)$, so $c_i=c(L_i)$.
For $0\le i\le J-2$, the last-step truncation does not occur and
$L_{i+1}-L_i=c_i/\sqrt{d\delta}$. Since $c$ is increasing and
$\log r\ge1-r^{-1}$ for $r>0$,
\[
  \sqrt{d\delta}\int_{L_i}^{L_{i+1}}\frac{\mathrm dL}{c(L)}
  +\log\frac{c_{i+1}}{c_i}
  \ge
  \sqrt{d\delta}\int_{L_i}^{L_{i+1}}\frac{\mathrm dL}{c_{i+1}}
  +\log\frac{c_{i+1}}{c_i}
  =
  \frac{c_i}{c_{i+1}}+\log\frac{c_{i+1}}{c_i}
  \ge1.
\]
Summing and adding the final step gives
\[
  J\le
  1+\sqrt{d\delta}\int_0^{d\lambda}\frac{\mathrm dL}{c(L)}
  +\log\frac{c(d\lambda)}{c(0)}.
\]
Since $c'(L)=1-e^{-c(L)}$ and
$(1-e^{-u})^{-1}\le1+u^{-1}$ for $u>0$ by $e^u\ge1+u$,
the substitution $u=c(L)$ gives
\[
\begin{aligned}
  \int_0^{d\lambda}\frac{\mathrm dL}{c(L)}
  &=\int_{c(0)}^{c(d\lambda)}
    \frac{\mathrm du}{u(1-e^{-u})}\\
  &\le\int_{c(0)}^{c(d\lambda)}
    \left(\frac1u+\frac1{u^2}\right)\mathrm du\\
  &=\log\frac{c(d\lambda)}{c(0)}
    +\frac1{c(0)}-\frac1{c(d\lambda)}
  \le\log\frac{c(d\lambda)}{c(0)}+\frac1{c(0)}.
\end{aligned}
\]
Consequently,
\[
  J\le
  1+\frac{\sqrt{d\delta}}{c(0)}
  +(1+\sqrt{d\delta})\log\frac{c(d\lambda)}{c(0)}.
\]
Finally, $c(0)=\log(1+\delta)$ and $\lambda\le\delta^{-1}$ give
\[
  c(d\lambda)
  =\log\left(1+e^{(d-1)\lambda}\right)
  \le\log2+(d-1)\lambda
  \le\frac d\delta.
\]
Substitution proves \eqref{eq:hs-J-simple}.
\end{proof}

Put
\begin{equation}
  A_\delta:=\sup_{L\ge0}\frac{L}{\log(1+\delta e^L)},
  \qquad
  F(t):=\frac{(1+t)^{1+t}}{t^t}\quad(t>0),\qquad F(0):=1.
  \label{eq:hs-A-F}
\end{equation}
Note that $F(t)$ is increasing since
$\frac{\mathrm d}{\mathrm dt}\log F(t)
=\log(1+t^{-1})>0$.

\begin{lemma}
\label{lem:hs-binomial-F}
For every $d\in\mathbb N$ and $m\in\mathbb N_0$,
\begin{equation}
  \binom{m+d}{d}\le F(m/d)^d.
  \label{eq:hs-binomial-F}
\end{equation}
\end{lemma}

\begin{proof}
For $m=0$, \eqref{eq:hs-binomial-F} is an equality. Let $m\ge1$. By the
binomial theorem,
\[
  1=
  \left(\frac{d}{m+d}+\frac{m}{m+d}\right)^{m+d}
  \ge
  \binom{m+d}{d}
  \left(\frac{d}{m+d}\right)^d
  \left(\frac{m}{m+d}\right)^m.
\]
Therefore
\[
  \binom{m+d}{d}
  \le
  \frac{(m+d)^{m+d}}{d^d m^m}
  =
  F(m/d)^d.
\]
\end{proof}

\begin{lemma}
\label{lem:hs-maximizer}
The supremum defining $A_\delta$ is attained uniquely, and
\begin{equation}
  F(A_\delta-1)=\delta^{-1}.
  \label{eq:hs-F-equation}
\end{equation}
Moreover, for every $h\ge0$,
\begin{equation}
  F(A_\delta+h)\le\delta^{-1}+3+3h.
  \label{eq:hs-F-simple}
\end{equation}
\end{lemma}

\begin{proof}
Since $F$ is continuous and strictly increasing from $1$ to infinity,
there is a unique $t>0$ with $F(t)=\delta^{-1}$.
The weighted arithmetic--geometric mean inequality gives, for $L\ge0$,
\[
\begin{aligned}
  1+\delta e^L
  &=\frac{t}{1+t}\frac{1+t}{t}
    +\frac1{1+t}(1+t)\delta e^L\\
  &\ge
  \left(\frac{1+t}{t}\right)^{t/(1+t)}
  \left((1+t)\delta e^L\right)^{1/(1+t)}
  =e^{L/(1+t)}.
\end{aligned}
\]
Thus $L/\log(1+\delta e^L)\le1+t$, with equality precisely when
$\delta e^L=1/t$. This occurs at
$L=\log(F(t)/t)>0$, since $F(t)>t$.
Hence $A_\delta=1+t$, the maximum is attained uniquely, and
\eqref{eq:hs-F-equation} follows.

We have
\[
  \frac{F''(t)}{F(t)}
  =\log^2(1+t^{-1})-\frac1{t(t+1)}.
\]
By Cauchy--Schwarz,
\[
  \log^2(1+t^{-1})
  =
  \left(\int_t^{t+1}\frac{\mathrm dx}{x}\right)^2
  <
  \left(\int_t^{t+1}1\,\mathrm dx\right)
  \left(\int_t^{t+1}\frac{\mathrm dx}{x^2}\right)
  =\frac1{t(t+1)}.
\]
Hence $F''(t)<0$, so $F$ is concave. Thus we have
\[
  F(1+h)-F(0)
  = F(1+h) - F(1) + 3
  \le F'(1)h + 3
  = (4 \log 2) h + 3
  \le 3h + 3.
\]
The increment of a concave function over an interval of fixed length decreases
as the interval is shifted to the right. Hence
\[
  F(A_\delta+h)-F(A_\delta-1)
  \le F(1+h)-F(0)\le3+3h.
\]
Together with \eqref{eq:hs-F-equation}, this proves
\eqref{eq:hs-F-simple}.
\end{proof}

\begin{corollary}
\label{cor:hs-simple-finite}
For every $d\in\mathbb N$ and $0<\delta<1$,
\begin{equation}
\begin{aligned}
  B_{[]}^{\rm sh}(d,\delta)
  &\le
  d+
  \left(
    1+\frac{\sqrt{d\delta}}{\log(1+\delta)}
    +(1+\sqrt{d\delta})
      \log\frac{d}{\delta\log(1+\delta)}
  \right)\\
  &\qquad\times
  \delta^{-d}
  \left(1+3\delta+3\sqrt{\frac\delta d}\right)^d.
\end{aligned}
  \label{eq:hs-simple-finite}
\end{equation}
\end{corollary}

\begin{proof}
The definition of $r_i^+$ and the shell recursion give
\[
  \frac{r_i^+}{d}
  <\frac{L_{i+1}}{c_i}
  \le\frac{L_i}{c_i}+\frac1{\sqrt{d\delta}}
  \le A_\delta+\frac1{\sqrt{d\delta}}.
\]
By \eqref{eq:hs-A-count},
Lemma~\ref{lem:hs-binomial-F}, the monotonicity of $F$,
and Lemma~\ref{lem:hs-maximizer}, we obtain
\begin{align}
  |\mathcal A_i|
  \le \binom{r_i^++d}{d}
  &\le  F\left(A_\delta+\frac1{\sqrt{d\delta}}\right)^d \label{eq:hs-shell-F-bound} \\
  &\le \left(\delta^{-1}+3+\frac3{\sqrt{d\delta}}\right)^d =
  \delta^{-d}
  \left(1+3\delta+3\sqrt{\frac\delta d}\right)^d. \notag
\end{align}
Summing over the $J$ shells and using \eqref{eq:hs-layer-finite} and
\eqref{eq:hs-J-simple} proves \eqref{eq:hs-simple-finite}.
\end{proof}

\begin{corollary}
\label{cor:hs-upper-base}
For every fixed $0<\delta<1$,
\begin{equation}
  \limsup_{d\to\infty}B_{[]}^{\rm sh}(d,\delta)^{1/d}
  \le
  \limsup_{d\to\infty}B_{\rm cov}^{\rm sh}(d,\delta)^{1/d}
  \le F(A_\delta).
  \label{eq:hs-upper-base}
\end{equation}
\end{corollary}

\begin{proof}
By \eqref{eq:hs-shell-F-bound} and
$B_{\rm cov}^{\rm sh}(d,\delta)\le2B_{[]}^{\rm sh}(d,\delta)$,
\[
  B_{\rm cov}^{\rm sh}(d,\delta)
  \le
  2d+2J F\left(A_\delta+\frac1{\sqrt{d\delta}}\right)^d.
\]
For fixed $\delta$, \eqref{eq:hs-J-simple} gives $J^{1/d}\to1$, and continuity
of $F$ yields the claimed upper bound. The first inequality follows from
$B_{[]}^{\rm sh}\le B_{\rm cov}^{\rm sh}$.
\end{proof}
\begin{remark}
Since $F(A_\delta-1)=\delta^{-1}$ and
$F(t)=e(t+1/2)+O(t^{-1})$ as $t\to\infty$,
\[
  F(A_\delta)=\delta^{-1}+e+O(\delta).
\]
Thus the upper exponential base in Corollary~\ref{cor:hs-upper-base} is
asymptotic to $\delta^{-1}+e$. For comparison, Gnewuch's refinement
\cite[Remark~2.10]{Gnewuch2024} is
\begin{equation}
  N_{[]}(d,\delta)
  \le
  \frac{d^d}{d!}
  \left(
  \frac{\delta^{-1}-1}{f_d(\delta)}
  +\frac12\left(1+\frac3d\right)
  \right)^d,
  \qquad
  f_d(\delta)
  :=\frac{d\bigl(1-(1-\delta)^{1/d}\bigr)}{\delta}.
  \label{eq:hs-gnewuch-refined}
\end{equation}
Its exponential base tends to
$e\delta^{-1}-e+O(\delta)$ as $d\to\infty$ and then $\delta\downarrow0$,
so the ratio of these two upper bases tends to $1/e$.
\end{remark}

\subsection{Fixed-dimensional asymptotics}
\label{sec:hs-fixed-finite}

The cumulative estimate used above deliberately ignores
the lower shell boundary. This is harmless for the high-dimensional
upper bound, but it is too crude for sharp fixed-dimensional
asymptotics. We now retain the actual shell thickness
$r_i^+-r_i^-+1$.

\begin{lemma}
\label{lem:hs-thin-shell}
\label{lem:hs-endpoint-layer}
For every $i$,
\begin{equation}
  |\mathcal A_i|
  \le
  (r_i^+-r_i^-+1)\binom{r_i^++d-1}{d-1},
  \label{eq:hs-thin-shell-count}
\end{equation}
and
\begin{equation}
  0\le
  |\mathcal A_i\cup(\mathcal A_i+\boldsymbol1)|-|\mathcal A_i|
  \le d\binom{r_i^++d-1}{d-1}.
  \label{eq:hs-endpoint-layer}
\end{equation}
\end{lemma}

\begin{proof}
The number of indices of total degree $r$ is
$\binom{r+d-1}{d-1}$, which is increasing in $r$. Summing over
$r_i^-\le r\le r_i^+$ gives \eqref{eq:hs-thin-shell-count}.

For the endpoint estimate, the shift by $\boldsymbol1$ can add indices only
on the $d$ layers of total degree $r_i^++1,\ldots,r_i^++d$. On the layer of
degree $r_i^++k$, the added indices arise from indices of $\mathcal A_i$ of
degree $r_i^++k-d\le r_i^+$, so there are at most
$\binom{r_i^++d-1}{d-1}$ of them. Summing over the $d$ possible layers gives
\eqref{eq:hs-endpoint-layer}.
\end{proof}

\begin{corollary}
\label{cor:hs-fixed-expansion}
For every fixed $d\ge2$, as $\delta\downarrow0$,
\begin{equation}
  B_{[]}^{\rm sh}(d,\delta)
  =(1+o_d(1))\delta^{-d},
  \qquad
  B_{\rm cov}^{\rm sh}(d,\delta)
  =(1+o_d(1))\delta^{-d}.
  \label{eq:hs-fixed-expansion}
\end{equation}
Thus the homothetic bracketing construction attains the sharp
fixed-dimensional leading constant, while its shell-wise endpoint count has
the same leading term.
\end{corollary}

\begin{proof}
Put
\[
  \tau:=\sqrt{\frac\delta d},\qquad x_i:=e^{-L_i}.
\]
Using $1-e^{-u}\le u$ and
$c_i=\log(1+\delta/x_i)\le\delta/x_i$,
\[
\begin{aligned}
  0\le x_i-x_{i+1}
  =x_i\left(1-e^{-(L_{i+1}-L_i)}\right)
  \le x_i\left(1-e^{-c_i/\sqrt{d\delta}}\right)
  \le \frac{x_i c_i}{\sqrt{d\delta}}
  \le \tau.
\end{aligned}
\] Lemma~\ref{lem:hs-J-simple} gives, for fixed $d$,
\[
  J\tau
  \le \tau+\frac{\delta}{\log(1+\delta)}
  +(\tau+\delta)\log\frac{d}{\delta\log(1+\delta)}
  =1+o_d(1).
\]
Hence, by telescoping,
\[
  0
  \le \sum_{i=0}^{J-1}\bigl(\tau-(x_i-x_{i+1})\bigr)
  =J\tau-(x_0-x_J)
  =J\tau-1+\delta^d
  =o_d(1).
\]
Since $\bigl(x\log(x^{-1})\bigr)^{d-1}$ is bounded on $[0,1]$,
\[
  \tau\sum_{i=0}^{J-1}\bigl(x_i\log(x_i^{-1})\bigr)^{d-1}
  \le
  \sum_{i=0}^{J-1}(x_i-x_{i+1})
  \bigl(x_i\log(x_i^{-1})\bigr)^{d-1}
  +o_d(1).
\]
The sum on the right is a Riemann sum with mesh at most $\tau\to0$.
Therefore
\[
  \limsup_{\delta\downarrow0}
  \tau\sum_{i=0}^{J-1}\bigl(x_i\log(x_i^{-1})\bigr)^{d-1}
  \le
  \int_0^1\bigl(x\log(x^{-1})\bigr)^{d-1}\,\mathrm dx
  =\frac{(d-1)!}{d^d}.
\]

The definitions of $r_i^\pm$ give uniformly in $i$
\[
  r_i^+-r_i^-+1\le \tau^{-1}+d+1,
\]
while $r_i^+<dL_{i+1}/c_i$ and
$c_i\ge\delta/(x_i+\delta)$ give
\[
  \frac{\delta(r_i^++d)}d
  \le x_i\log(x_i^{-1})
  +\tau+\delta+d\delta\log(\delta^{-1})
  =x_i\log(x_i^{-1})+o_d(1).
\]
Adding the two inequalities in Lemma~\ref{lem:hs-thin-shell} and using
the preceding bounds, we obtain, uniformly in $i$,
\[
\begin{aligned}
  \delta^d|\mathcal A_i\cup(\mathcal A_i+\boldsymbol1)|
  &\le
  \delta^d(r_i^+-r_i^-+1+d)\binom{r_i^++d-1}{d-1}\\
  &\le
  \frac{\delta^d}{(d-1)!}
  (\tau^{-1}+2d+1)(r_i^++d)^{d-1}\\
  &\le
  \frac{d^d}{(d-1)!}
  \left(\tau+\frac{2d+1}{d}\delta\right)
  \bigl(x_i\log(x_i^{-1})+o_d(1)\bigr)^{d-1}.
\end{aligned}
\]
Since $J\tau=O_d(1)$ and $\delta/\tau\to0$, the uniform
$o_d(1)$ contributes only $o_d(1)$ after summation. The preceding
Riemann-sum estimate therefore gives
\[
  \limsup_{\delta\downarrow0}
  \delta^d B_{\rm cov}^{\rm sh}(d,\delta)\le1.
\]
Finally, $1\le\delta^dN_{[]}(d,\delta)
\le\delta^d B_{[]}^{\rm sh}(d,\delta)
\le\delta^d B_{\rm cov}^{\rm sh}(d,\delta)$,
so both equalities in \eqref{eq:hs-fixed-expansion} follow at once.
\end{proof}

\begin{remark}
The limiting integral is exactly the radial density of the probability
measure $\pi_d$ used in the lower bound. With $x=e^{-L}$,
\[
  \frac{d^d}{(d-1)!}
  \int_0^1\bigl(x\log(x^{-1})\bigr)^{d-1}\,\mathrm dx
  =
  \frac{d^d}{(d-1)!}
  \int_0^\infty L^{d-1}e^{-dL}\,\mathrm dL
  =1.
\]
\end{remark}

\section{Bounds for star discrepancy}
\label{sec:discrepancy}

We now prove Theorem~\ref{thm:discrepancy-23463}. The proof uses the finite
bracketing estimates from Corollary~\ref{cor:simple-bounds} and
Corollary~\ref{cor:hs-simple-finite}, together with Gnewuch's general bound
\eqref{eq:gnewuch2024}. Following Wei\ss{} \cite[proof of Theorem~1]{Weiss2026},
we use Hammersley point sets in dimensions three and four for the large-$n$
range. All $d\ge5$ are handled uniformly.

\subsection{A cover-based chaining estimate}
\label{sec:discrepancy-chaining}

The dyadic chaining arguments used in
\cite{Aistleitner2011,AistleitnerHofer2014,GPW2021} can be abstracted in
terms of the cover cardinalities at the individual levels as follows.

\begin{theorem}
\label{thm:independent-chaining}
Let $d,n,H\ge1$ and $1\le\mu\le H$ be integers. For $\mu\le m\le H$, put
\[
  C_m:=
  \begin{cases}
    \min\{N_{[]}(d,2^{-m}),N(d,2^{-m})\},&m<H,\\
    N(d,2^{-H}),&m=H.
  \end{cases}
\]
Choose positive weights $w_{\mathcal B}$ and $w_h$, $\mu<h\le H$, such
that
\[
  w_{\mathcal B}+\sum_{h=\mu+1}^{H}w_h<1,
\]
and set
\[
  L_{\mathcal B}:=\log\frac{2C_\mu}{w_{\mathcal B}},
  \qquad
  L_h:=\log\frac{2C_h}{w_h}.
\]
Then there exists an $n$-point set $P\subset[0,1)^d$ satisfying
\begin{align}
  nD_n^*(P)
  \le
  \sqrt{\frac n2L_{\mathcal B}}
  +
  \sum_{h=\mu+1}^{H}
  \left(
    \sqrt{2n2^{-(h-1)}L_h}
    +\frac13L_h
  \right)
  +n2^{-H}.
  \label{eq:independent-chaining}
\end{align}
\end{theorem}

\begin{proof}
For $m=\mu,\ldots,H-1$, choose $\Gamma_m$ as follows. If
$N_{[]}(d,2^{-m})\le N(d,2^{-m})$, let $\Gamma_m$ be the set of lower
endpoints of a minimum $2^{-m}$-bracketing cover. Otherwise let $\Gamma_m$
be a minimum $2^{-m}$-cover. In either case $|\Gamma_m|\le C_m$. Fix a map
\[
  \boldsymbol\ell_m:[0,1]^d\longrightarrow
  \Gamma_m\cup\{\boldsymbol0\}
\]
such that, for every $\boldsymbol y\in[0,1]^d$,
\[
  \boldsymbol\ell_m(\boldsymbol y)\le\boldsymbol y,
  \qquad
  V(\boldsymbol y)-V(\boldsymbol\ell_m(\boldsymbol y))\le2^{-m}.
\]
Let $\Gamma_H$ be a minimum $2^{-H}$-cover, so $|\Gamma_H|=C_H$.

For a measurable $A\subset[0,1)^d$, write
\[
  S(A):=\sum_{i=1}^n
  \bigl(\mathbf1_A(\boldsymbol X_i)-\lambda_d(A)\bigr),
\]
where $\boldsymbol X_1,\ldots,\boldsymbol X_n$ are independent and uniformly
distributed on $[0,1)^d$. For every $\boldsymbol y$, the terminal cover gives
$\boldsymbol v,\boldsymbol w\in\Gamma_H\cup\{\boldsymbol0\}$ with
$\boldsymbol v\le\boldsymbol y\le\boldsymbol w$ and
$V(\boldsymbol w)-V(\boldsymbol v)\le2^{-H}$. Since the contribution at
$\boldsymbol0$ is zero,
\begin{equation}
  |S([\boldsymbol0,\boldsymbol y))|
  \le
  \max_{\boldsymbol z\in\Gamma_H}
  |S([\boldsymbol0,\boldsymbol z))|+n2^{-H}.
  \label{eq:discrepancy-terminal-comparison}
\end{equation}

For $\boldsymbol z\in\Gamma_H$, put
$\boldsymbol\beta_H(\boldsymbol z)=\boldsymbol z$ and recursively define
\[
  \boldsymbol\beta_h(\boldsymbol z)
  =\boldsymbol\ell_h(\boldsymbol\beta_{h+1}(\boldsymbol z)),
  \qquad h=H-1,\ldots,\mu.
\]
Define
$B_\mu(\boldsymbol z)
  :=[\boldsymbol0,\boldsymbol\beta_\mu(\boldsymbol z)),$
and, for $\mu<h\le H$,
\[
  A_h(\boldsymbol z)
  :=[\boldsymbol0,\boldsymbol\beta_h(\boldsymbol z))
  \setminus
  [\boldsymbol0,\boldsymbol\beta_{h-1}(\boldsymbol z)).
\]
Then we have a disjoint union
\[
  [\boldsymbol0,\boldsymbol z)
  =B_\mu(\boldsymbol z)\,\cup\,
  \mathop{\bigcup}_{h=\mu+1}^{H}A_h(\boldsymbol z).
\]
Moreover,
\[
  \lambda_d(A_h(\boldsymbol z))
  =V(\boldsymbol\beta_h(\boldsymbol z))
   -V(\boldsymbol\beta_{h-1}(\boldsymbol z))
  \le2^{-(h-1)}.
\]
Let $\mathcal B_\mu$ and $\mathcal A_h$ be the families of distinct nonempty
sets $B_\mu(\boldsymbol z)$ and $A_h(\boldsymbol z)$, respectively, as
$\boldsymbol z$ ranges over $\Gamma_H$. Since
$\boldsymbol\beta_h(\boldsymbol z)\in\Gamma_h\cup\{\boldsymbol0\}$, and
$A_h(\boldsymbol z)$ is determined by $\boldsymbol\beta_h(\boldsymbol z)$,
we have
\[
  |\mathcal B_\mu|\le C_\mu,
  \qquad
  |\mathcal A_h|\le C_h.
\]
Combining the decomposition with
\eqref{eq:discrepancy-terminal-comparison} gives
\begin{equation}
  nD_n^*
  \le
  \max_{B\in\mathcal B_\mu}|S(B)|
  +\sum_{h=\mu+1}^{H}\max_{A\in\mathcal A_h}|S(A)|
  +n2^{-H}.
  \label{eq:discrepancy-finite-chain}
\end{equation}

Hoeffding's inequality gives, for each $B\in\mathcal B_\mu$,
\[
  \mathbb P\left(
    |S(B)|>\sqrt{\frac n2L_{\mathcal B}}
  \right)
  \le 2e^{-L_{\mathcal B}}
  =\frac{w_{\mathcal B}}{C_\mu}.
\]
For a measurable $A$ with $\lambda_d(A)\le p$, Bernstein's inequality
\cite[Theorem~2.10]{BLM2013} in the form
\[
  \mathbb P\left(
    |S(A)|>\sqrt{2npL}+\frac L3
  \right)
  \le2e^{-L}
\]
holds for every $L>0$. Hence, for $A\in\mathcal A_h$,
\[
  \mathbb P\left(
    |S(A)|>
    \sqrt{2n2^{-(h-1)}L_h}+\frac13L_h
  \right)
  \le 2e^{-L_h}
  =\frac{w_h}{C_h}.
\]
The union bound therefore gives total failure probability strictly smaller
than one. Since the sampled points are pairwise distinct almost surely,
a realization consisting of $n$ distinct points and satisfying
\eqref{eq:independent-chaining} exists.
\end{proof}

When $\mu=H$, the sum is empty. Letting
$w_{\mathcal B}\uparrow1$ gives the single-scale consequence
\begin{equation}
  D^*(n,d)
  \le
  \sqrt{\frac{1}{2n}\log\bigl(2N(d,2^{-H})\bigr)}
  +2^{-H}.
  \label{eq:discrepancy-single-scale}
\end{equation}

\subsection{Proof of Theorem~\ref{thm:discrepancy-23463}}
\label{sec:discrepancy-proof}

\begin{proof}
Throughout the proof, we write each estimate in the form
\[
  D^*(n,d)\le C\sqrt{\frac dn}
\]
and refer to $C$ as the normalized constant.

\medskip
\noindent\textbf{Case $d=1$.}
The $n$-point midpoint grid has star discrepancy $1/(2n)$. Hence
\[
  D^*(n,1)\le\frac1{2n}
  \le0.5\sqrt{\frac1n}.
\]
Thus this case gives normalized constant at most $0.5$.

\medskip
\noindent\textbf{Case $d=2$.}
For $n\ge2$, the base-$2$ Hammersley point set satisfies
\cite[proof of Theorem~1]{Weiss2026}
\[
  D_n^*\le
  \frac{7}{2n}+\frac{\log n}{(2\log 2)n}
  =\frac{7+\log_2 n}{2n}
  =\frac{7+\log_2 n}{2\sqrt{2n}}\sqrt{\frac2n}.
\]
The coefficient $(7+\log_2 n)/(2\sqrt{2n})$ is decreasing for $n\ge2$
by differentiation and equals $2$ at $n=2$. Therefore
\[
  D^*(n,2)\le2\sqrt{\frac2n}
  \qquad(n\ge2).
\]
The same estimate is trivial for $n=1$. Thus this case gives normalized
constant $2$.

\medskip
\noindent\textbf{Case $d=3$.}
For $n\ge256$, use the Hammersley point set obtained from the Halton sequence
in bases $2$ and $3$. By \cite[proof of Theorem~1]{Weiss2026},
\[
  D_n^*\bigl(\mathrm{Ham}_{2,3}(n)\bigr)
  \le
  \frac1n\left[
    3+
    \left(\frac{\log n}{2\log2}+\frac32\right)
    \left(\frac{\log n}{\log3}+2\right)
  \right].
\]
After normalization by $\sqrt{3/n}$, the right-hand side is a positive
linear combination of $n^{-1/2}$, $(\log n)n^{-1/2}$, and
$(\log n)^2n^{-1/2}$, all decreasing for $n\ge e^4$. Hence, since
$256>e^4$, it suffices to evaluate the bound at $n=256$. The normalized
constant in this range is therefore less than
\[
  \frac{3+(11/2)(8\log2/\log3+2)}{16\sqrt3}
  <1.50692.
\]

For $1\le n\le255$, \eqref{eq:discrepancy-single-scale} with $H=4$
and \eqref{eq:exponential-cover} give
\[
  \frac{D^*(n,3)}{\sqrt{3/n}}
  \le
  \sqrt{\frac{13\log2}{6}+\frac{\sqrt3}{4}}
  +\frac1{16}\sqrt{\frac n3}
  \le
  \sqrt{\frac{13\log2}{6}+\frac{\sqrt3}{4}}
  +\frac{\sqrt{85}}{16}
  <1.97.
\]
Combining the two ranges, this case gives normalized constant less than $1.97$.

\medskip
\noindent\textbf{Case $d=4$.}
For $n\ge2^{17}$, use the Hammersley point set in the bases $2,3,5$.
Atanassov's estimate \cite{Atanassov2004}, in the form stated in
\cite[Lemma~5]{Weiss2026}, gives
\[
\begin{aligned}
  D_n^*\bigl(\mathrm{Ham}_{2,3,5}(n)\bigr)
  \le\frac1n\Bigg[&
  \frac16
  \left(\frac{\log n}{\log2}+3\right)
  \left(\frac{\log n}{\log3}+3\right)
  \left(\frac{2\log n}{\log5}+3\right)\\
  &+3\left(\frac{\log n}{\log2}+1\right)
  +\frac52
  \left(\frac{\log n}{\log2}+2\right)
  \left(\frac{\log n}{\log3}+2\right)+3
  \Bigg].
\end{aligned}
\]
After division by $\sqrt{4/n}$, the right-hand side is a positive linear
combination of $n^{-1/2}(\log n)^k$, $0\le k\le3$. These functions are
decreasing when $\log n\ge6$, so it suffices to consider $n=2^{17}$.
Since $\log 2/\log 3 < 2/3$ and $\log 2/\log 5 < 1/2$,
the normalized bound at $n=2^{17}$ is less than
\[
  \frac{1}{512\sqrt2}
  \left[
    \frac16\,20\cdot\frac{43}{3}\cdot20
    +3\cdot18
    +\frac52\,19\cdot\frac{40}{3}+3
  \right]
  =\frac{14813}{4608\sqrt2}
  <2.27309.
\]

For $1\le n\le2^{17}$, \eqref{eq:discrepancy-single-scale} with $H=9$
and \eqref{eq:exponential-cover} give
\[
  \frac{D^*(n,4)}{\sqrt{4/n}}
  \le
  \sqrt{\frac{37\log2}{8}+\frac{\sqrt2}{16}}
  +\frac{\sqrt n}{2^{10}}
  \le
  \sqrt{\frac{37\log2}{8}+\frac{\sqrt2}{16}}
  +2^{-3/2}
  <2.17.
\]
Combining the two ranges, this case gives normalized constant less than $2.27309$.

\medskip
\noindent\textbf{Case $d\ge5$.}
We first isolate the two numerical covering estimates needed below. We claim
\begin{align}
  \frac1{2d}\log\bigl(4N_{[]}(d,2^{-14})\bigr)
  &<5.004,
  \label{eq:discrepancy-single-scale-cardinality} \\
  \frac1{2d}\log\left(
    \frac{2N_{[]}(d,2^{-12})}{0.9}\right)
  &<4.264.
  \label{eq:discrepancy-coarse-cardinality}
\end{align}
We verify both estimates with the same dimension split.
For $5\le d\le64$, \eqref{eq:exponential-bracketing} gives
\[
\begin{aligned}
  N_{[]}(d,2^{-h})
  &\le d+2^{hd}\exp\left(d\sqrt{2d\,2^{-h}}\right)\\
  &\le2^{hd}\exp\left(d\sqrt{2d\,2^{-h}}\right)
    \left(
      1+d\,2^{-5h}\exp\left(-d\sqrt{2d\,2^{-h}}\right)
    \right).
\end{aligned}
\]
Using $\log(1+x)\le x$ and
$\exp(-d\sqrt{2d\,2^{-h}})\le1$, we bound the left-hand sides in
\eqref{eq:discrepancy-single-scale-cardinality} and
\eqref{eq:discrepancy-coarse-cardinality} by
\[
  \frac12\left(
    h\log2+\sqrt{2d\,2^{-h}}
    +2^{-5h}+\frac{\log\kappa}{d}
  \right),
  \qquad
  (h,\kappa)=(14,4),\ (12,2/0.9),
\]
respectively. As a function of $u=\sqrt d$, the varying part has the convex
form $au+c/u^2$ with $a,c>0$, so it suffices to check $d=5$ and $d=64$.
The endpoint pairs are below $(5.004,4.908)$ and $(4.264,4.255)$,
respectively.

For $d\ge65$, \eqref{eq:hs-J-simple} gives $J\le17d$ for
$\delta=2^{-12},2^{-14}$. Indeed, after division by $d$,
\[
  \frac Jd
  \le
  \frac1d+\frac{\sqrt{\delta/d}}{\log(1+\delta)}
  +\left(\frac1d+\sqrt{\frac\delta d}\right)
    \log\frac{d}{\delta\log(1+\delta)}.
\]
For fixed $\delta$, the first two terms decrease with $d$, while the
remaining terms are constant multiples of $\log(Cd)/d$ and
$\log(Cd)/\sqrt d$, with $C=1/(\delta\log(1+\delta))$.
These functions decrease when $\log(Cd)>1$ and $\log(Cd)>2$,
respectively, as is the case here. Thus it suffices to evaluate the
right-hand side at $d=65$, where
\[
  \frac1{65}
  +\frac{\sqrt{\delta/65}}{\log(1+\delta)}
  +\left(\frac1{65}+\sqrt{\frac\delta{65}}\right)
    \log\frac{65}{\delta\log(1+\delta)}
  <
  \begin{cases}
    8.316,&\delta=2^{-12},\\
    16.278,&\delta=2^{-14}.
  \end{cases}
\]
In both cases the bound is less than $17$.
Corollary~\ref{cor:hs-simple-finite} therefore gives, for
$h\in\{12,14\}$,
\[
  N_{[]}(d,2^{-h})
  \le
  18d\,2^{hd}
  \left(1+3\,2^{-h}+3\sqrt{\frac{2^{-h}}d}\right)^d.
\]
Hence the left-hand sides in
\eqref{eq:discrepancy-single-scale-cardinality} and
\eqref{eq:discrepancy-coarse-cardinality} are bounded by
\[
  \frac12\left[
    h\log2
    +\log\left(1+3\,2^{-h}+3\sqrt{\frac{2^{-h}}d}\right)
    +\frac{\log(Kd)}d
  \right]
\]
with $(h,K)=(14,72)$ and $(12,40)$, respectively.
Both bounds decrease with $d$, and at $d=65$ they are below
$4.919$ and $4.223$. This proves
\eqref{eq:discrepancy-single-scale-cardinality} and
\eqref{eq:discrepancy-coarse-cardinality} for every $d\ge5$.

\medskip
\noindent\emph{Subcase $n<2^{21}d$.}
Apply \eqref{eq:discrepancy-single-scale} with $H=14$. Since
$N(d,2^{-14})\le2N_{[]}(d,2^{-14})$ and
$\sqrt{n/d}<2^{21/2}$,
\[
  \frac{D^*(n,d)}{\sqrt{d/n}}
  <
  \sqrt{\frac1{2d}\log\bigl(4N_{[]}(d,2^{-14})\bigr)}
  +2^{-7/2}
  <2.237+2^{-7/2}
  <2.326
\]
since \eqref{eq:discrepancy-single-scale-cardinality} bounds the square-root
term by $\sqrt{5.004}<2.237$.

\medskip
\noindent\emph{Subcase $n\ge2^{21}d$.}
Put $t=n/d$ and take in Theorem~\ref{thm:independent-chaining}
\[
  \mu=12,
  \qquad
  w_{\mathcal B}=0.9,
  \qquad
  w_h=\frac3{100}\left(\frac23\right)^{h-13},
\]
with
\[
  H:=\left\lceil\frac12\log_2t\right\rceil+7.
\]
Then $H\ge18$ and
\[
  0.9+\sum_{h=13}^{\infty}w_h=0.99<1,
  \qquad
  2^{H-8}<\sqrt t\le2^{H-7}.
\]
With $L_{\mathcal B}$ and $L_h$ as in
Theorem~\ref{thm:independent-chaining}, dividing
\eqref{eq:independent-chaining} by $\sqrt{nd}=d\sqrt t$ gives
\begin{align}
  \frac{D^*(n,d)}{\sqrt{d/n}}
  &\le
  \sqrt{\frac{L_{\mathcal B}}{2d}}
  +\sum_{h=13}^{H}\sqrt{\frac{2^{2-h}L_h}{d}}
  +
  \frac1{3d\sqrt t}\sum_{h=13}^{H}L_h
  +\sqrt t\,2^{-H}.
  \label{eq:discrepancy-normalized-chaining}
\end{align}
Since $C_{12}\le N_{[]}(d,2^{-12})$, the first term on the right-hand side
of \eqref{eq:discrepancy-normalized-chaining} satisfies
\[
  \sqrt{\frac{L_{\mathcal B}}{2d}}
  \le
  \sqrt{\frac1{2d}\log\frac{2N_{[]}(d,2^{-12})}{0.9}}
  <\sqrt{4.264}.
\]

To estimate the second term on the right-hand side of
\eqref{eq:discrepancy-normalized-chaining}, set
\[
  \Lambda_h:=
  h\log2+1+
  \frac15\left(
    \log\frac{200}{3}+(h-13)\log\frac32
  \right)
  \qquad(h\ge13).
\]
We use Gnewuch's bound at these levels because it gives an affine majorant
in $h$ and allows a simple estimate of the resulting series.
Indeed, \eqref{eq:gnewuch2024} and
$\log(d!)\ge d\log d-d+1$ give
\[
  \frac1d\log\frac{2N_{[]}(d,2^{-h})}{w_h}
  \le\Lambda_h,
  \qquad
  \frac1d\log\frac{4N_{[]}(d,2^{-h})}{w_h}
  \le\Lambda_h+\frac{\log2}{5}.
\]
The first estimate is used for $13\le h<H$. At $h=H$, use
$C_H=N(d,2^{-H})\le2N_{[]}(d,2^{-H})$ and the second estimate. Hence the
middle two terms on the right-hand side of
\eqref{eq:discrepancy-normalized-chaining} are at most
\begin{align}
  &\sum_{h=13}^{H-1}\sqrt{2^{2-h}\Lambda_h}
  +\sqrt{2^{2-H}\left(\Lambda_H+\frac{\log2}{5}\right)}
  +
  \frac1{3\cdot2^{H-8}}
  \left(
    \sum_{h=13}^{H}\Lambda_h+\frac{\log2}{5}
  \right) \notag\\
  &\le
  \sum_{h=13}^{H}\sqrt{2^{2-h}\Lambda_h}
  +\sqrt{2^{2-H}\frac{\log2}{5}}
  +
  \frac1{3\cdot2^{H-8}}
  \left(
    \sum_{h=13}^{H}\Lambda_h+\frac{\log2}{5}
  \right),
\label{eq:discrepancy-large-increments}
\end{align}
where we used $\sqrt{x+y}\le\sqrt x+\sqrt y$.

Write $\Lambda_h=ah+b$, where
\[
  a=\log2+\frac15\log\frac32,
  \qquad
  b=1+\frac15\left(\log\frac{200}{3}-13\log\frac32\right).
\]
Since $0<a<0.835$ and $0.835\cdot13-\Lambda_{13}>0$,
the affine function $0.835h-\Lambda_h$ is positive for $h\ge13$.
Hence $\Lambda_h<0.835h$.
The last two terms in \eqref{eq:discrepancy-large-increments} are therefore
at most
\[
  2^{-H/2}\left[
    2\sqrt{\frac{\log2}{5}}
    +\frac{256}{3}2^{-H/2}
      \left(0.835\sum_{h=13}^{H}h+\frac{\log2}{5}\right)
  \right].
\]
The bracket decreases for $H\ge18$, since its second term has successive
ratio at most
\[
  2^{-1/2}\left(1+\frac{H+1}{\sum_{h=13}^{H}h}\right)
  \le2^{-1/2}\left(1+\frac{19}{93}\right)<1.
\]
Its value at $H=18$ is less than $13.72$. Thus the sum of the last two
terms in \eqref{eq:discrepancy-large-increments} is less than
$13.72 \cdot 2^{-H/2}$.
On the other hand, $\Lambda_h$ is increasing and
$\Lambda_H\ge H\log2\ge18\log2$, so
\[
\begin{aligned}
  \sum_{h=H+1}^{\infty}\sqrt{2^{2-h}\Lambda_h}
  \ge
  \sqrt{\Lambda_H}\sum_{h=H+1}^{\infty}2^{1-h/2}
  &= \frac{2\sqrt{\Lambda_H}}{\sqrt2-1}\,2^{-H/2}\\
  &\ge
  \frac{2\sqrt{18\log2}}{\sqrt2-1}\,2^{-H/2}\\
  &>17\,2^{-H/2} > 13.72\cdot 2^{-H/2}.
\end{aligned}
\]
Therefore the right-hand side of
\eqref{eq:discrepancy-large-increments} is less than
\[
  \sum_{h=13}^{\infty}\sqrt{2^{2-h}\Lambda_h}
  <
  \frac{\sqrt{0.835\cdot13\,2^{-11}}}
       {1-2^{-1/2}\sqrt{14/13}}
  <0.2735,
\]
since the ratio of two successive terms in the majorant
$\sqrt{0.835h\,2^{2-h}}$ is at most
$2^{-1/2}\sqrt{14/13}$.
Finally, the last term in
\eqref{eq:discrepancy-normalized-chaining} satisfies
$\sqrt t\,2^{-H}\le2^{-7}$. Consequently,
\[
  \frac{D^*(n,d)}{\sqrt{d/n}}
  <\sqrt{4.264}+0.2735+2^{-7}
  <2.3463.
\]
Combining the two subcases, this case gives normalized constant less than $2.3463$.
This proves \eqref{eq:discrepancy-23463}.

Finally, \eqref{eq:inverse-55052} follows by taking
$n=\lceil5.5052d\varepsilon^{-2}\rceil$, since $2.3463^2<5.5052$.
\end{proof}

\section*{Acknowledgments}
The author thanks Michael Gnewuch for many helpful discussions and for generously sharing his expertise on bracketing and $\delta$-covering problems and the related literature.
\section*{Declaration of generative AI and AI-assisted technologies in the manuscript preparation process}
The author used ChatGPT 5.6 Sol for literature searches, exploratory
development, mathematical discussion, and assistance in preparing portions
of the exposition and LaTeX source, including the TikZ code for
Figures~\ref{fig:construction-d2} and~\ref{fig:homothetic-log-shell}.
All mathematical arguments, calculations,
references, and conclusions were independently checked and verified by the
author, who takes full responsibility for the contents of the paper.


\begin{thebibliography}{99}
\footnotesize
\sloppy

\bibitem{Aistleitner2011}
C.~Aistleitner,
\newblock Covering numbers, dyadic chaining and discrepancy,
\newblock \emph{Journal of Complexity} \textbf{27} (2011), 531--540.
\newblock \href{https://doi.org/10.1016/j.jco.2011.03.001}{doi:10.1016/j.jco.2011.03.001}.

\bibitem{AistleitnerHofer2014}
C.~Aistleitner and M.~Hofer,
\newblock Probabilistic discrepancy bound for Monte Carlo point sets,
\newblock \emph{Mathematics of Computation} \textbf{83} (2014), 1373--1381.
\newblock \href{https://doi.org/10.1090/S0025-5718-2013-02773-1}{doi:10.1090/S0025-5718-2013-02773-1}.

\bibitem{Atanassov2004}
E.~I.~Atanassov,
\newblock On the discrepancy of the Halton sequences,
\newblock \emph{Mathematica Balkanica (N.S.)} \textbf{18} (2004), 15--32.

\bibitem{BLM2013}
S.~Boucheron, G.~Lugosi, and P.~Massart,
\newblock \emph{Concentration Inequalities: A Nonasymptotic Theory of Independence},
\newblock Oxford University Press, Oxford, 2013.
\newblock \href{https://doi.org/10.1093/acprof:oso/9780199535255.001.0001}{doi:10.1093/acprof:oso/9780199535255.001.0001}.

\bibitem{DickL1_2026}
J.~Dick,
\newblock The $L_1$-discrepancy with nonnegative weights suffers from the curse of dimensionality,
\newblock preprint, 2026.
\newblock \href{https://arxiv.org/abs/2607.24290}{arXiv:2607.24290}.

\bibitem{DickStar2026}
J.~Dick,
\newblock A Proof of the Novak--Wo\'zniakowski Conjecture: Optimal Polynomial Tractability Exponents for the Inverse Star Discrepancy,
\newblock preprint, 2026.
\newblock \href{https://arxiv.org/abs/2607.23571}{arXiv:2607.23571}.

\bibitem{DickPillichshammer2010}
J.~Dick and F.~Pillichshammer,
\newblock \emph{Digital Nets and Sequences: Discrepancy Theory and Quasi-Monte Carlo Integration},
\newblock Cambridge University Press, Cambridge, 2010.
\newblock \href{https://doi.org/10.1017/CBO9780511761188}{doi:10.1017/CBO9780511761188}.

\bibitem{DickRudolfZhu2016}
J.~Dick, D.~Rudolf, and H.~Zhu,
\newblock Discrepancy bounds for uniformly ergodic Markov chain quasi-Monte Carlo,
\newblock \emph{Annals of Applied Probability} \textbf{26} (2016), 3178--3205.
\newblock \href{https://doi.org/10.1214/16-AAP1173}{doi:10.1214/16-AAP1173}.

\bibitem{DGS2005}
B.~Doerr, M.~Gnewuch, and A.~Srivastav,
\newblock Bounds and constructions for the star-discrepancy via $\delta$-covers,
\newblock \emph{Journal of Complexity} \textbf{21} (2005), 691--709.
\newblock \href{https://doi.org/10.1016/j.jco.2005.05.002}{doi:10.1016/j.jco.2005.05.002}.

\bibitem{DoerrGnewuchWahlstrom2010}
B.~Doerr, M.~Gnewuch, and M.~Wahlstr\"om,
\newblock Algorithmic construction of low-discrepancy point sets via dependent randomized rounding,
\newblock \emph{Journal of Complexity} \textbf{26} (2010), 490--507.
\newblock \href{https://doi.org/10.1016/j.jco.2010.03.004}{doi:10.1016/j.jco.2010.03.004}.

\bibitem{Gnewuch2008JCo}
M.~Gnewuch,
\newblock Bracketing numbers for axis-parallel boxes and applications to geometric discrepancy,
\newblock \emph{Journal of Complexity} \textbf{24} (2008), 154--172.
\newblock \href{https://doi.org/10.1016/j.jco.2007.08.003}{doi:10.1016/j.jco.2007.08.003}.

\bibitem{Gnewuch2008EJC}
M.~Gnewuch,
\newblock Construction of minimal bracketing covers for rectangles,
\newblock \emph{Electronic Journal of Combinatorics} \textbf{15} (2008), Article R95.
\newblock \href{https://doi.org/10.37236/819}{doi:10.37236/819}.

\bibitem{Gnewuch2012Survey}
M.~Gnewuch,
\newblock Entropy, randomization, derandomization, and discrepancy,
\newblock in L.~Plaskota and H.~Wo\'zniakowski (eds.),
\emph{Monte Carlo and Quasi-Monte Carlo Methods 2010},
Springer Proceedings in Mathematics \& Statistics, vol.~23,
Springer, Heidelberg, 2012, pp.~43--78.
\newblock \href{https://doi.org/10.1007/978-3-642-27440-4_3}{doi:10.1007/978-3-642-27440-4\_3}.

\bibitem{GnewuchWahlstromWinzen2012}
M.~Gnewuch, M.~Wahlstr\"om, and C.~Winzen,
\newblock A new randomized algorithm to approximate the star discrepancy based on threshold accepting,
\newblock \emph{SIAM Journal on Numerical Analysis} \textbf{50} (2012), 781--807.
\newblock \href{https://doi.org/10.1137/110833865}{doi:10.1137/110833865}.

\bibitem{Gnewuch2024}
M.~Gnewuch,
\newblock Improved bounds for the bracketing number of orthants or revisiting an algorithm of Thi\'emard to compute bounds for the star discrepancy,
\newblock \emph{Journal of Complexity} \textbf{83} (2024), Article 101855.
\newblock \href{https://doi.org/10.1016/j.jco.2024.101855}{doi:10.1016/j.jco.2024.101855}.

\bibitem{GnewuchHebbinghaus2021}
M.~Gnewuch and N.~Hebbinghaus,
\newblock Discrepancy bounds for a class of negatively dependent random points including Latin hypercube samples,
\newblock \emph{Annals of Applied Probability} \textbf{31} (2021), 1944--1965.
\newblock \href{https://doi.org/10.1214/20-AAP1638}{doi:10.1214/20-AAP1638}.

\bibitem{GPW2021}
M.~Gnewuch, H.~Pasing, and C.~Wei\ss,
\newblock A generalized Faulhaber inequality, improved bracketing covers, and applications to discrepancy,
\newblock \emph{Mathematics of Computation} \textbf{90} (2021), 2873--2898.
\newblock \href{https://doi.org/10.1090/mcom/3666}{doi:10.1090/mcom/3666}.

\bibitem{Hinrichs2004}
A.~Hinrichs,
\newblock Covering numbers, Vapnik--{\v C}ervonenkis classes and bounds for the star-discrepancy,
\newblock \emph{Journal of Complexity} \textbf{20} (2004), 477--483.
\newblock \href{https://doi.org/10.1016/j.jco.2004.01.001}{doi:10.1016/j.jco.2004.01.001}.

\bibitem{HNWW2001}
S.~Heinrich, E.~Novak, G.~W.~Wasilkowski, and H.~Wo\'zniakowski,
\newblock The inverse of the star-discrepancy depends linearly on the dimension,
\newblock \emph{Acta Arithmetica} \textbf{96} (2001), 279--302.
\newblock \href{https://doi.org/10.4064/aa96-3-7}{doi:10.4064/aa96-3-7}.

\bibitem{Niederreiter1992}
H.~Niederreiter,
\newblock \emph{Random Number Generation and Quasi-Monte Carlo Methods},
\newblock CBMS--NSF Regional Conference Series in Applied Mathematics, vol.~63,
Society for Industrial and Applied Mathematics, Philadelphia, PA, 1992.
\newblock \href{https://doi.org/10.1137/1.9781611970081}{doi:10.1137/1.9781611970081}.

\bibitem{NovakPillichshammer2026}
E.~Novak and F.~Pillichshammer,
\newblock Upper bounds for generalized $L_p$-discrepancy of random points,
\newblock \emph{Journal of Complexity} \textbf{95} (2026), Article 102028.
\newblock \href{https://doi.org/10.1016/j.jco.2026.102028}{doi:10.1016/j.jco.2026.102028}.

\bibitem{Thiemard2001}
E.~Thi\'emard,
\newblock An algorithm to compute bounds for the star discrepancy,
\newblock \emph{Journal of Complexity} \textbf{17} (2001), 850--880.
\newblock \href{https://doi.org/10.1006/jcom.2001.0600}{doi:10.1006/jcom.2001.0600}.

\bibitem{Weiss2026}
C.~Wei\ss,
\newblock Hammersley point sets and inverse of star-discrepancy,
\newblock \emph{Journal of Complexity} \textbf{93} (2026), Article 101998.
\newblock \href{https://doi.org/10.1016/j.jco.2025.101998}{doi:10.1016/j.jco.2025.101998}.

\end{thebibliography}
\end{document}